\documentclass{amsart}
\usepackage[T1]{fontenc}
\usepackage[english]{babel}
\usepackage{amsmath, amsthm, amsfonts, amssymb, amsfonts}
\usepackage{mathtools}
\usepackage{lmodern}
\usepackage{comment}
\numberwithin{equation}{section}

\usepackage[left=3.5cm,right=3.5cm,top=3cm,bottom=3cm]{geometry}

\title{An elementary proof of a criterion for subfunctors of EXT to be closed}

\author{Juan Camilo Cala}\thanks{The author was supported by CONAHCyT}
\address{Departamento de Matemáticas, Facultad de Ciencias, Universidad Nacional Autónoma de México, Circuito Exterior s/n, Ciudad Universitaria, CP 04510, Ciudad de México, México}
\email{jccalab@gmail.com}

\usepackage[square, numbers, comma, sort&compress]{natbib} % Use the natbib reference package - read up on this to edit the reference style; if you want text (e.g. Smith et al., 2012) for the in-text references (instead of numbers), remove 'numbers'

\usepackage[backref=page]{hyperref}
\hypersetup{colorlinks=true,linkcolor=blue,urlcolor=blue,citecolor=blue} 

\usepackage{tikz-cd} %\begin{tikzcd}
\newcommand{\mc}[1]{\mathcal{#1}}
\newcommand{\A}{\mc{A}}

\newcommand{\C}{\mc{C}}
\newcommand{\D}{\mc{D}}
\newcommand{\E}{\mc{E}}

\newcommand{\op}{\mathrm{op}}
\newcommand{\eps}{\varepsilon}

\DeclareMathOperator{\Kker}{Ker}
\newcommand\Ker[1]{\Kker({#1})}

\DeclareMathOperator{\CCoker}{Coker}
\newcommand\Coker[1]{\CCoker({#1})}

\DeclareMathOperator{\Iimg}{Im}
\newcommand\Img[1]{\Iimg({#1})}
\newcommand{\sets}{\mathsf{Set}}
\newcommand{\ab}{\mathsf{Ab}}

\let\hom\relax% Set equal to \relax so that LaTeX thinks it's not defined
\DeclareMathOperator{\hhom}{Hom}
\newcommand\hom[3]{\hhom_{#1}\left({#2},{#3}\right)}

\DeclareMathOperator{\Ext}{Ext}
\newcommand\ext[3]{\Ext^{1}_{#1}\left({#2},{#3}\right)} %Ext(-,-)
\usepackage{aliascnt}%paquete para cambiar el nombre en autoref por el problema con Teorema para citar teoremas

\newtheorem{teo}{Theorem}[section]

\newaliascnt{lema}{teo}
\newtheorem{lema}[lema]{Lemma}
\aliascntresetthe{lema}

\newaliascnt{prop}{teo}
\newtheorem{prop}[prop]{Proposition}
\aliascntresetthe{prop}

\newaliascnt{coro}{teo}
\newtheorem{coro}[coro]{Corollary}
\aliascntresetthe{coro}

\theoremstyle{definition}
\newaliascnt{defi}{teo}
\newtheorem{defi}[defi]{Definition}
\aliascntresetthe{defi}

\newaliascnt{ejem}{teo}

\aliascntresetthe{ejem}

\newaliascnt{obs}{teo}
\newtheorem{obs}[obs]{Remark}
\aliascntresetthe{obs}

\begin{document}

\begin{abstract}
Let $\A$ be an abelian category and let $F$ be a subbifunctor of the additive bifunctor $\ext{\A}{-}{-}\colon \A^{\op}\times \A\to \ab$. Buan proved in \citep{buan_closed_2001} that $F$ is closed if, and only if, $F$ has the $3\times 3$-lemma property, a certain diagrammatic property satisfied by the class of $F$-exact sequences. The proof of this result relies on the theory of exact categories and on the Freyd--Mitchell embedding theorem, a very well-known overpowered result. In this paper we provide a proof of Buan's result only by means of elementary methods in abelian categories. To achieve this we survey the required theory of subfunctors leading us to a self-contained exposition of this topic.
\end{abstract}

\maketitle

%====== Introduction ======%

\section{Introduction}
In \citep{butler_classes_1961} Butler and Horrocks introduced the notion of Ext-subbifunctors over abelian categories under the name of natural classes of simple extensions and E-functors. In this work, they also defined closed subfunctors and proved that these were intimately related to certain classes of morphisms, called h.f.\,classes, that were previously introduced by Buchsbaum in \citep{buchsbaum_note_1959}. This was an interesting fact because Buchsbaum had already shown \citep{buchsbaum_note_1959, buchsbaum_satellites_1960} that a theory of relative homological algebra can be developed with this kind of classes. From this it became clear that the work of Butler and Horrocks gave the very first insight of how relative homological algebra could be improved by the study of the theory of subfunctors of Ext, an idea that was later explored and formalized by Auslander and Solberg \citep{auslander_relative1_1993, auslander_relative2_1993, auslander_relative3_1993} in the context of categories of modules over an Artin algebra.

Later on, Dräxler et al.\ \citep{draxler_exact_1999} studied closed subfunctors and the relation with their collections of exact sequences, but now in the context of exact categories in the sense of Quillen \citep{quillen_higher_1973}, which are a natural generalization of abelian categories. One of their main results says that the definition of closed subfunctor given by Butler and Horrocks is redundant \citep[Proposition 1.4]{draxler_exact_1999}.

After that, Buan \citep{buan_closed_2001} showed that a subfunctor is triangulated if, and only if, it is closed. For this, he proved that closed subfunctors are exactly those whose induced collection of exact sequences satisfies a certain $3\times 3$-lemma property. Nevertheless, its proof is based on the fact that the collection of exact sequences induced by a subfunctor defines an exact structure over the underlying abelian category and so, by a result of Keller \citep[Appendix A]{keller_chain_1990}, there is a version of the Freyd--Mitchell embedding theorem \citep[Theorem 7.34]{freyd_abelian_1964} that applies to this type of category, that is, there exists an exact embedding sending the exact sequences of the induced collection to short exact sequences in some abelian category. Therefore the problem becomes trivial because the $3\times 3$-lemma holds in abelian categories (see for instance \citep[Exercise 1.3.2]{weibel_introduction_1994}).

In this paper we present a self-contained exposition of the theory of Ext-subfunctors over abelian categories following the ideas of Butler and Horrocks \citep{butler_classes_1961}, Auslander and Solberg \citep{auslander_relative1_1993}, and Dräxler et al.\ \citep{draxler_exact_1999}. The objective behind this is to establish a clean path for proving the mentioned result of Buan \citep{buan_closed_2001} without using the theory of exact categories or the Freyd--Mitchell embedding theorem. We will write full proofs of some of the results using more modern notation, especially the ones contained in the work of Butler and Horrocks \citep{butler_classes_1961}.

\subsection*{Organization of the paper.} We now describe how we will proceed to achieve our goal. In Section \ref{section-2} we first recall Yoneda's construction of the Ext-bifunctor over abelian categories without assuming the existence of enough projectives or injectives. Then we recall the notion of a subfunctor of Ext and establish the connection between subfunctors and their induced collections of short exact sequences in \autoref{prop-biyec-subf-clas}. We end this section by introducing the concept of proper functors in \autoref{def-proper-subf} and by stating equivalent conditions for a subfunctor of Ext to be proper in \autoref{prop-carac-subf-propio}. 

In Section \ref{section-3} we start by giving the axioms defining f.\,classes and h.f.\,classes. We see how one can construct a proper subfunctor from an f.\,class and \textit{vice versa}, and in \autoref{teo-caract-subfpropio-fclases} we show that both constructions are mutually inverse. After that, we recall the notion of closed subfunctors and see that these correspond to h.f.\,classes under the bijection between proper subfunctors and f.\,classes that we mentioned before. Then, we state and prove the result of Buan in \autoref{teo-principal} using elementary methods in abelian categories and, as a direct consequence, we derive the fact that the definition of closed subfunctor is redundant. Lastly, we show how \autoref{teo-principal} can be applied to easily deduce that certain subfunctors are closed, making use of the fact that the $3\times 3$-lemma holds in abelian categories.

\subsection*{Conventions.} Throughout this paper, $\A$  will denote an abelian category and $\A^{\op}$ its opposite category, that is, the category whose objects are the same as those of $\A$ and whose arrows are given by reversing the arrows of $\A$. We write $\operatorname{Mor}(\A)$ to refer to the class of morphisms in $\A$. We use the notation $A\in\A$ to say that $A$ is an object of $\A$. Given $A,B\in\A$, $\hom{\A}{A}{B}$ is the abelian group of morphisms $f\colon A\to B$ in $\A$. We also denote by $\ab$ and $\sets$ the category of abelian groups and the category of sets, respectively.

\section{Subfunctors of Ext}\label{section-2}
\subsection{Yoneda's Ext construction}
We recall the Yoneda construction \citep{yoneda_homology_1954} of the $\operatorname{Ext}$-bifunctor. Most of the material cover here is taken from Mitchell's book \citep{mitchell_theory_1965}.

Consider two objects $A,C\in\A$. We denote by $\mathcal{E}_{\A}(C,A)$ the collection of all short exact sequences in $\A$ of the form $$\eps\colon\ \ \  0\longrightarrow A\overset{i}{\longrightarrow} B\overset{p}{\longrightarrow} C\longrightarrow 0\,.$$ Then we denote by $\mathcal{E}_{\A}$ the collection of all short exact sequences in $\A$. These are the objects of a category (which we denote it in the same way) where a morphism $(f,g,h)\colon\eps\to\eta$ between two exact sequences $\eps\in\E_{A}(C,A)$ and $\eta\in\E_{\A}(Z,X)$ is a commutative diagram $$\begin{tikzcd}
	{\phantom{\eta}\eps\colon\ \ \ 0} & A & B & C & {0\,\phantom{.}} \\
	{\phantom{\eps}\eta\colon\ \ \ 0} & X & Y & Z & {0\,.}
	\arrow[from=1-1, to=1-2]
	\arrow["i", from=1-2, to=1-3]
	\arrow["f", from=1-2, to=2-2]
	\arrow["p", from=1-3, to=1-4]
	\arrow["g", from=1-3, to=2-3]
	\arrow[from=1-4, to=1-5]
	\arrow["h", from=1-4, to=2-4]
	\arrow[from=2-1, to=2-2]
	\arrow["j"', from=2-2, to=2-3]
	\arrow["q"', from=2-3, to=2-4]
	\arrow[from=2-4, to=2-5]
\end{tikzcd}$$ Composition in $\E_{\A}$ is defined component-wise and the identity morphisms are the obvious ones. Then, a morphism $(f,g,h)\colon\eps\to\eta$ is an isomorphism in $\E_{\A}$ if, and only if, $f,g$ and $h$ are isomorphisms in $\A$. With this information it is easy to see that $\E_{\A}$ is an additive category. For example, for $k=1,2$, let $\eps_k\in\E_{\A}(C_k,A_k)$ be given by $$\begin{tikzcd}
	{\eps_k\colon\ \ \ 0} & {A_k} & {B_k} & {C_k} & 0\, .
	\arrow[from=1-1, to=1-2]
	\arrow["{i_k}", from=1-2, to=1-3]
	\arrow["{p_k}", from=1-3, to=1-4]
	\arrow[from=1-4, to=1-5]
\end{tikzcd}$$ Then the direct sum $\eps_1\oplus\eps_2$ is the short exact sequence \begin{equation}\label{eq-direct-sum}
\begin{tikzcd}
	{0} & {A_1\oplus A_2} & {B_1\oplus B_2} & {C_1\oplus C_2} & 0\,.
	\arrow[from=1-1, to=1-2]
	\arrow["{i_1\oplus i_2}", from=1-2, to=1-3]
	\arrow["{p_1\oplus p_2}", from=1-3, to=1-4]
	\arrow[from=1-4, to=1-5]
\end{tikzcd}
\end{equation}

Two short exact sequences $\eps$ and $\eps^{\prime}$ in $\E_{\A}(C,A)$ are \textit{Yoneda equivalent} if there exists a morphism $(1_A,g,1_C)\colon \eps\to\eps^{\prime}$, that is, a commutative diagram $$\begin{tikzcd}
	{\phantom{{}^{\prime}}\eps\colon\ \ \ 0} & A & B & C & {0\,\phantom{.}} \\
	{\eps^{\prime}\colon\ \ \ 0} & A & {B^{\prime}} & C & {0\,.}
	\arrow[from=1-1, to=1-2]
	\arrow["i", from=1-2, to=1-3]
	\arrow[Rightarrow, no head, from=1-2, to=2-2]
	\arrow["p", from=1-3, to=1-4]
	\arrow["g", dashed, from=1-3, to=2-3]
	\arrow[from=1-4, to=1-5]
	\arrow[Rightarrow, no head, from=1-4, to=2-4]
	\arrow[from=2-1, to=2-2]
	\arrow["{i^{\prime}}"', from=2-2, to=2-3]
	\arrow["{p^{\prime}}"', from=2-3, to=2-4]
	\arrow[from=2-4, to=2-5]
\end{tikzcd}$$ We see that $g$ is necessarily an isomorphism due to the Five Lemma and therefore the previous construction defines an equivalence relation on $\E_{\A}(C,A)$. We denote by $\ext{\A}{C}{A}$ the quotient class and its elements will be referred as $[\eps]$ for representative $\eps\in \E_{\A}(C,A)$. For example, any split short exact sequence $\eps\in \E_{\A}(C,A)$ satisfies $[\eps]=[\eps_{C,A}]$, where we put \begin{equation}\label{eq-split-exact}
\begin{tikzcd}[ampersand replacement=\&]
	{\eps_{C,A}\colon\ \ \ 0} \& A \& {A\oplus C} \& C \& 0\, .
	\arrow["\mu_A", from=1-2, to=1-3]
	\arrow["\pi_C", from=1-3, to=1-4]
	\arrow[from=1-1, to=1-2]
	\arrow[from=1-4, to=1-5]
\end{tikzcd}
\end{equation} Here $\mu_A$ and $\pi_C$ denote the canonical inclusion and projection of the direct sum, respectively.

In order to see that the construction described above is functorial, we define for an exact sequence $\eps\in \E_{\A}(C,A)$ and a morphism $f\in\hom{A}{A}{A^{\prime}}$ the correspondence
\begin{align*}
\ext{\A}{C}{f}\colon \ext{\A}{C}{A}&\to\ext{\A}{C}{A^{\prime}}\\
[\eps]&\mapsto [f\cdot\eps],
\end{align*}
where $f\cdot \eps$ denotes the \textit{pushout of $\eps$ along $f$}, that is, an exact sequence such that in the following commutative diagram the left-sided square is a pushout $$\begin{tikzcd}
	{\phantom{f\cdot}\eps\colon\ \ \ 0} & A & B & C & {0\,\phantom{.}} \\
	{f\cdot\eps\colon\ \ \ 0} & {A^{\prime}} & {B^{\prime}} & C & {0\,.}
	\arrow[from=1-1, to=1-2]
	\arrow["i", from=1-2, to=1-3]
	\arrow["f"', from=1-2, to=2-2]
	\arrow["\textsc{po}"{anchor=center}, draw=none, from=1-2, to=2-3]
	\arrow["p", from=1-3, to=1-4]
	\arrow["l", from=1-3, to=2-3]
	\arrow[from=1-4, to=1-5]
	\arrow[Rightarrow, no head, from=1-4, to=2-4]
	\arrow[from=2-1, to=2-2]
	\arrow["j"', from=2-2, to=2-3]
	\arrow["q"', from=2-3, to=2-4]
	\arrow[from=2-4, to=2-5]
\end{tikzcd}$$ Dually we define for an exact sequence $\eps\in \E_{\A}(C,A)$ and a morphism $g\in\hom{\A}{C^{\prime}}{C}$ the correspondence \begin{align*}
\ext{\A}{g}{A}\colon \ext{\A}{C}{A}&\to\ext{\A}{C^{\prime}}{A}\\
[\eps]&\mapsto [\eps\cdot g],
\end{align*}
where $\eps\cdot g$ denotes the \textit{pullback of $\eps$ along $g$}. This information constitutes a bifunctor $\ext{\A}{-}{-}\colon \A^{\op}\times\A\to\sets$ contravariant in the first variable and covariant in the second variable. Some of the properties satisfied by this construction are listed below.

\begin{prop}\label{prop-propiedades-pullback-pushout}
Any morphism $(f,g,h)\colon \eps\to\eps^{\prime}$ between short exact sequences $\eps\in\E_{\A}(C,A)$ and $\eps^{\prime}\in\E_{\A}(C^{\prime},A^{\prime})$ admits a factorization $$\left(\begin{tikzcd}[column sep = 3em]
	{\eps} & {\eps^{\prime}}
	\arrow["{(f,g,h)}", from=1-1, to=1-2]
\end{tikzcd}\right)=\left(\begin{tikzcd}[column sep = 3.5em]
	{\eps} & {\overline{\eps}} & {\eps^{\prime}}
	\arrow["{(f,\overline{g},1_C)}", from=1-1, to=1-2]
	\arrow["{(1_{A^{\prime}},g^{\prime},h)}", from=1-2, to=1-3]
\end{tikzcd}\right)$$ where $\overline{\eps}\in\E_{\A}(C,A^{\prime})$, and this implies the relations $[f \cdot \eps] = [\overline{\eps}] = [\eps^{\prime} \cdot h]$. Therefore, for any short exact sequence $\eps\in\E_{\A}(C,A)$ the following properties holds true:
\begin{itemize}
\item[$(a)$] $[1_A \cdot \varepsilon]=[\varepsilon]= [\varepsilon \cdot 1_C]$.
\item[$(b)$] $[(f^{\prime} f) \cdot \varepsilon]=[f^{\prime} \cdot(f \cdot \varepsilon)]$, for all $A \overset{f}{\rightarrow} A^{\prime} \overset{f^{\prime}}{\rightarrow} A^{\prime \prime}$.
\item[$(c)$] $[\varepsilon \cdot(gg^{\prime})]=[(\varepsilon \cdot g) \cdot g^{\prime}]$, for all $C^{\prime \prime} \overset{g^{\prime}}{\rightarrow} C^{\prime} \overset{g}{\rightarrow} C$.
\item[$(d)$] $[(f \cdot \varepsilon) \cdot g]=[f \cdot(\varepsilon \cdot g)]$, for all $f\colon A\to A^{\prime}$ and $g\colon C^{\prime}\to C$.
\item[$(e)$] For every $X\in\A$, $[0_{A,X}\cdot\eps]=[\eps_{C,A}]=[\eps\cdot 0_{X,C}]$, where $0_{A,X}\colon A	\to X$ and $0_{X,C}\colon X\to C$ are zero morphisms and $\eps_{C,A}$ is given by \eqref{eq-split-exact}.\hfill $\qed$
\end{itemize}
\end{prop}

We can endow each $\ext{A}{C}{A}$ with an additive structure that turn it into an abelian group. This is done by taking the \textit{Baer's sum}: for given $\eps,\eps^{\prime}\in \E_{\A}(C,A)$ define \begin{equation}\label{eq-baer-sum}
[\eps]+[\eps^{\prime}]\coloneqq [\nabla_A \cdot(\eps\oplus\eps^{\prime})\cdot\Delta_C],
\end{equation} where $\eps\oplus \eps^{\prime}$ is the direct sum as defined in \eqref{eq-direct-sum}, $\Delta_C\colon C\to C\oplus C$ is the \textit{diagonal morphism} and $\nabla_A\colon A\oplus A\to A$ is the \textit{codiagonal morphism}, each of which is completely determined by its matricial representation $\nabla_A=\begin{psmallmatrix}
1_A & 1_A
\end{psmallmatrix}$ and $\Delta_C=\begin{psmallmatrix}
1_C \\ 1_C
\end{psmallmatrix}$.  The zero element of the abelian group $\ext{\A}{C}{A}$ is $[\eps_{C,A}]$, the class of split short exact sequences. Therefore we obtain an additive bifunctor $\ext{\A}{-}{-}\colon \A^{\op}\times \A\to\ab$. 

Finally, for given $\eps\in \E_{\A}(C,A)$ and $X\in \A$ there is an induced morphism of abelian groups, known as the \textit{covariant connecting morphism}, given by \begin{equation}\label{eq-covariant-connecting} \begin{aligned}
\partial^{\eps}_{X}\colon \hom{\A}{X}{C}&\to \ext{\A}{X}{A}\\
f&\mapsto [\eps\cdot f].
\end{aligned}\end{equation} Dually we have the \textit{contravariant connecting morphism} given by \begin{equation}\label{eq-contravariant-connecting} \begin{aligned}
\delta_{X}^{\eps}\colon \hom{\A}{A}{X}&\to \ext{\A}{C}{X}\\
f&\mapsto [f\cdot\eps].
\end{aligned} \end{equation} 

\begin{teo}\label{teo-ext-seq}
For $\eps\in \E_{\A}(C,A)$ and $X\in\A$ the following holds:
\begin{itemize}
\item[$(a)$] The covariant connecting morphism $\partial_{X}^{\eps}$ given in \eqref{eq-covariant-connecting} is natural in $X$. In addition, any morphism $(f,g,h)\colon\eps\to\eta$, with $\eta\in\E_{\A}(C^{\prime},A^{\prime})$, gives rise to a commutative square $$\begin{tikzcd}
	{\hom{\A}{X}{C}} & {\ext{\A}{X}{A}\phantom{.}} \\
	{\hom{\A}{X}{C^{\prime}}} & {\ext{\A}{X}{A^{\prime}}.}
	\arrow["{\hom{\A}{X}{h}}"', from=1-1, to=2-1]
	\arrow["{\ext{\A}{X}{f}}", from=1-2, to=2-2]
	\arrow["{\partial_X^{\eps}}", from=1-1, to=1-2]
	\arrow["{\partial_X^{\eta}}"', from=2-1, to=2-2]
\end{tikzcd}$$ Dually, the same holds for the contravariant connecting morphism $\delta_{X}^{\eps}$ given in \eqref{eq-contravariant-connecting}.
\item[$(b)$]  If $\eps$ is given by $$\begin{tikzcd}[column sep = 1.5em]
	{\varepsilon\colon\ \ \ 0} & A & B & C & 0\, ,
	\arrow[from=1-1, to=1-2]
	\arrow["i", from=1-2, to=1-3]
	\arrow["p", from=1-3, to=1-4]
	\arrow[from=1-4, to=1-5]
\end{tikzcd}$$ then there is an induced long exact sequence of abelian groups $$\begin{tikzcd}[column sep = 1em]
	0 & {\hom{\A}{X}{A}} & {\hom{\A}{X}{B}} & {\hom{\A}{X}{C}} \\
	& {\ext{\A}{X}{A}} & {\ext{\A}{X}{B}} & {\ext{\A}{X}{C}.}
	\arrow[from=1-1, to=1-2]
	\arrow[from=1-2, to=1-3]
	\arrow[from=1-3, to=1-4]
	\arrow[from=2-3, to=2-4]
	\arrow[from=2-2, to=2-3]
	\arrow["{\partial_X^{\eps}}"', from=1-4, to=2-2,out=-8, in=172]{dll}
\end{tikzcd}$$ Dually, there is an induced long exact sequence of abelian groups with the contravariant connecting morphism $\delta_{X}^{\eps}$.\hfill $\qed$
\end{itemize}
\end{teo}

\subsection{Subfunctors and exact sequences} 
In what follows we first recall the definition of an Ext-subbifunctor and the properties satisfied by its induced class of short exact sequences. Then we introduce the notion of proper functor.

In a categorical setting, a \textit{subfunctor} of a functor $G\colon \C\to \D$ between two categories $\C$ and $\D$ is a subobject $F$ of $G$ in the category of functors (here we assume that $\C$ is small and so the functor category is indeed a category). This means that a subfunctor of $G$ is a pair $(F,\alpha)$ consisting of a functor $F\colon \C\to \D$ and a natural transformation $\alpha\colon F\to G$ such that its components $\alpha_X\colon FX\to GX$ are monic for every $X\in\C$. When the source category is given by a product of two categories, a subfunctor is simply called a \textit{subbifunctor}. In our setting, an \textit{Ext-subbifunctor} is merely a subfunctor $(F,\alpha)$ of $\ext{\A}{-}{-}\colon \A^{\op}\times \A\to\sets$ and is completely determined by the following data: 
\begin{itemize}
\item[(SF1)] The natural transformation $\alpha$ can be taken as the set inclusion. This means that for every pair of objects $C,A\in\A$, $F(C,A)\subseteq \ext{\A}{C}{A}$. 
\item[(SF2)] For every $C,A\in\A$ the induced functors $F(C,-)$ and $F(-,A)$, both together with the natural inclusion, defines subfunctors of the corresponding induced functors $\ext{\A}{C}{-}\colon \A\to\sets$ and $\ext{\A}{-}{A}\colon \A^{\op}\to\sets$, respectively. Thereby any two morphisms $f\in\hom{\A}{C^{\prime}}{C}$ and $g\in\hom{\A}{A}{A^{\prime}}$ gives rise to a commutative square $$\begin{tikzcd}
	{F(C,A)} & {\ext{\A}{C}{A}}\, \\
	{F(C^{\prime},A^{\prime})} & {\ext{\A}{C^{\prime}}{A^{\prime}}.}
	\arrow["\subseteq", from=1-1, to=1-2]
	\arrow["{F(f,g)}"', from=1-1, to=2-1]
	\arrow["{\ext{\A}{f}{g}}", from=1-2, to=2-2]
	\arrow["\subseteq", from=2-1, to=2-2]
\end{tikzcd}$$ In other words, the action of $F$ over morphisms in $\A$ is given by restricting the action of the $\operatorname{Ext}$ functor over the elements of $F$. 
\end{itemize} Notice that in the previous definition we are considering subfunctors when the target category is $\sets$. In considering subbifunctors of the additive bifunctor $\ext{\A}{-}{-}\colon\A^{\op}\times\A\to\ab$, we will refer to them as \textit{additive subbifunctors}. The reason of why we adopt this terminology is for avoiding confusion due to the well-known fact that, in general, subfunctors of an additive functor are additive.

\begin{lema}\label{lema-subf-addfunc}
Every subfunctor of an additive functor $G\colon \A\to\ab$ is additive.\hfill $\qed$
\end{lema}

From now on, by a subfunctor we always mean a subbifunctor of the functor $\ext{\A}{-}{-}\colon \A^{\op}\times\A\to\sets$, and by an additive subfunctor we always mean a subbifunctor of the additive functor $\ext{\A}{-}{-}\colon\A^{\op}\times\A\to\ab$.

Every subfunctor $F$ has associated a collection of short exact sequences in $\A$. Indeed, for every pair of objects $C,A\in\A$, let $\E_F(C,A)\coloneqq\{\eps\in\E_{\A}(C,A)\colon [\eps]\in F(C,A)\}$ and then denote by $\E_F$ the collection of all short exact sequences arising in this form. Elements in $\mathcal{E}_F$ are called \textit{$F$-exact sequences}. Notice that the collection of exact sequences associated to the $\operatorname{Ext}$-functor is just $\E_{\A}$.

\begin{defi}
For a given a class of short exact sequences $\E\subseteq \E_{\A}$, we say that:
\begin{itemize}
    \item[(a)] $\E$ is \textit{closed under pushouts} if given a morphism $f\in\hom{\A}{A}{X}$ and a short exact sequence $\eps\in\E(C,A)$, any representative of $[f\cdot\eps]$ belongs to $\E$.
    \item[(b)] $\E$ is \textit{closed under pullbacks} if given a morphism $g\in\hom{\A}{Y}{C}$ and a short exact sequence $\eps\in\E(C,A)$, any representative of $[\eps\cdot g]$ belongs to $\E$.
    \item[(c)] $\E$ is \textit{closed under Baer sums} if given short exact sequences $\eps_1,\eps_2\in\E(C,A)$, any representative of the Baer sum $[\eps_1]+[\eps_2]$ as defined in \eqref{eq-baer-sum} belongs to $\E$.
    \item[(d)] $\E$ is \textit{closed under (finite) direct sums} if given short exact sequences $\eps_1,\eps_2\in\E$, the direct sum $\eps_1\oplus\eps_2$ as defined in \eqref{eq-direct-sum} belongs to $\E$.
    \item[(e)] $\E$ is \textit{closed under isomorphisms} if for any isomorphism $(f,g,h)\colon\eps\to\eta$ with $\eps\in\E$, then $\eta\in\E$.
    \item[(f)] $\E$ is \textit{closed under direct summands} if for $\eps_1,\eps_2\in\E_{\A}$ such that $\eps_1\oplus\eps_2\in\E$, then $\eps_1,\eps_2\in\E$.
\end{itemize}
\end{defi}

Let $\E_0\subseteq \E_{\A}$ be the collection of all split short exact sequences in $\A$. That is, every sequence in $\E_0$ is Yoneda-equivalent to some $\eps_{C,A}$ as defined in \eqref{eq-split-exact}. We clearly see that $\E_0$ is closed under pushouts, pullbacks, finite direct sums, Baer's sums and isomorphisms. Actually, $\E_0$ is the smallest non-empty collection closed under pushouts and pullbacks.

\begin{lema}\label{lema-eps0-eps}
If $\E\subseteq \E_{\A}$ is a non-empty collection of short exact sequences closed under pushouts and pullbacks, then $\E_0\subseteq\E$.
\end{lema}

\begin{proof}
If there is some $\eps\in\E(C,A)$, then by \autoref{prop-propiedades-pullback-pushout} we have $[\eps_{X,Y}]=[0_{Y,A}\cdot\eps\cdot 0_{X,C}]$, for every $X,Y\in\A$. Since $\E$ is closed under pushouts and pullbacks, $\eps_{X,Y}\in\E$. Thus $\E_0\subseteq \E$.
\end{proof}

It is straightforward to check that if $F$ is a subfunctor then the collection $\E_F$ of $F$-exact sequences is closed under pushouts and pullbacks. Moreover, if $F$ is an additive subfunctor, then $\E_F$ is closed under Baer sums. The following criterion for decide whether a given subfunctor is additive was first established by Auslander and Solberg in \citep{auslander_relative1_1993}.

\begin{prop}[{\citep[Lemma 1.1]{auslander_relative1_1993}}] \label{prop-caract-subf-addtv}
Let $F$ be a subfunctor. Then, $F$ is an additive subfunctor if, and only if, the collection $\E_F$ of $F$-exact sequences is non-empty and closed under direct sums. \hfill $\qed$
\end{prop}

Reciprocally, any collection $\E\subseteq\E_{\A}$ of short exact sequences which is closed under pushouts and pullbacks gives rise to a subfunctor $F_{\E}$ defined as follows: on objects, for $A,C\in\A$, $F_{\E}(C,A)$ consists of Yoneda equivalence classes $[\eps]$ whose representatives are short exact sequences $\eps\in\E(C,A)$, and on morphisms the action is given by restricting the action of the functor $\ext{\A}{-}{-}$. If in addition $\E$ is non-empty and closed under direct sums, then $F_{\E}$ is an additive subfunctor due to \autoref{prop-caract-subf-addtv}. This construction for passing from subfunctors to collections of short exact sequences and \textit{vice versa} actually defines a bijection between these two classes. We record this observation formally as follows.

\begin{prop}\label{prop-biyec-subf-clas}
There is a bijection between the following two classes:
\begin{itemize}
\item[($a$)] Subfunctors $F$ of $\ext{\A}{-}{-}\colon\A^{\op}\times\A\to\sets$.
\item[($b$)] Collections $\E\subseteq \E_{\A}$ of short exact sequences which are closed under pushouts and pullbacks.
\end{itemize}
The bijection is given by $F\mapsto \E_F$ and its inverse is given by $\E \mapsto F_{\E}$. Under this bijection, additive subfunctors correspond to non-empty collections of short exact sequences which are additionally closed under direct sums. \hfill $\qed$
\end{prop}

For example, the additive subfunctor which sends everything to zero corresponds via the bijection of \autoref{prop-biyec-subf-clas} to the collection of all split short exact sequences $\E_0$. 

In view of \autoref{prop-biyec-subf-clas} and \autoref{prop-caract-subf-addtv} we derive the following.

\begin{lema}\label{lema-equiv-clscf-clsbaer}
If $\E\subseteq \E_{\A}$ is a collection of short exact sequences closed under pushouts and pullbacks, then:
\begin{itemize}
\item[$(a)$] $\E$ is closed under direct summands. In particular, $\E$ is closed under isomorphisms.
\item[$(b)$] The following statements are equivalent if $\E$ is non-empty:
\begin{itemize}
\item[$(b1)$] $\E$ is closed under finite direct sums.
\item[$(b2)$] $\E$ is closed under Baer's sums.
\end{itemize}
\end{itemize}
\end{lema}

\begin{proof}
($a$)\ For $j=1,2$, consider $\eps_j\in\E_{\A}(C_j,A_j)$ such that $\eps=\eps_1\oplus\eps_2\in\E$. Then $[\eps_j]=[\pi_j^{A}\cdot\eps\cdot\mu_j^{C}]$, where we set $A\coloneqq A_1\oplus A_2$, $C\coloneqq C_1\oplus C_2$, and $\pi_j^{A}\colon A\to A_j$ and $\mu_j^{C}\colon C_j\to C$ are the canonical projections and inclusions, respectively. Since $\E$ is closed under pushouts and pullbacks it follows that $\eps_j\in\E$ for $j=1,2$.

%Es sabido que la existencia de un morfismo $(f,g,h)\colon \xi\to\eta$ entre sucesiones exactas cortas implica $[f\cdot\xi]=[\eta\cdot h]$. De manera que si $(f,g,h)\colon\xi\to\eta$ es un isomorfismo con $\xi\in\E$, como $\E$ es cerrada por pushouts, $f\cdot \xi\sim\eta\cdot h\in\E$, y como $\E$ es cerrada por pullbacks, $$\eta\sim\eta\cdot 1_C\sim\eta\cdot(hh^{-1})\sim(\eta\cdot h)\cdot h^{-1}\in\E.$$
($b$)\ $\E$ is closed under direct sums if, and only if, $F_{\E}$ is additive by \autoref{prop-caract-subf-addtv} and the bijection of \autoref{prop-biyec-subf-clas}. Then it is clear that $F_{\E}$ is additive if, and only if, $\E$ is closed under Baer's sums, and so the result follows.
\end{proof}

One more thing we can say about a subfunctor $F$ is that the images of the connecting morphisms \eqref{eq-covariant-connecting} and \eqref{eq-contravariant-connecting} associated to an $F$-exact sequence are again $F$-exact sequences and therefore, by \autoref{teo-ext-seq}\,$(b)$, we obtain an exact sequence as follows. 

\begin{prop}[{\citep[Proposition 1.3]{auslander_relative1_1993}}]\label{prop-subf-seq}
    Let $F$ be a subfunctor and let $\eps\in\E_F(C,A)$ be an $F$-exact sequence of the form $$\begin{tikzcd}[column sep = 1.5em]
	{\eps\colon\ \ \ 0} & A & B & C & 0\,.
	\arrow[from=1-1, to=1-2]
	\arrow["f", from=1-2, to=1-3]
	\arrow["g", from=1-3, to=1-4]
	\arrow[from=1-4, to=1-5]
\end{tikzcd}$$ Then, for every $X\in\A$ the sequences $$\begin{tikzcd}[column sep = 1.5em]
	0 & {\hom{\A}{C}{X}} & {\hom{\A}{B}{X}} & {\hom{\A}{A}{X}} & {F(C,X)}
	\arrow[from=1-1, to=1-2]
	\arrow[from=1-2, to=1-3]
	\arrow[from=1-3, to=1-4]
	\arrow["{\partial_X^{\eps}}", from=1-4, to=1-5]
\end{tikzcd}$$ and $$\begin{tikzcd}[column sep = 1.5em]
	0 & {\hom{\A}{X}{A}} & {\hom{\A}{X}{B}} & {\hom{\A}{X}{C}} & {F(X,A)}
	\arrow[from=1-1, to=1-2]
	\arrow[from=1-2, to=1-3]
	\arrow[from=1-3, to=1-4]
	\arrow["{\delta_X^{\eps}}", from=1-4, to=1-5]
\end{tikzcd}$$ are exact. \hfill $\qed$
\end{prop}
\begin{comment}
\begin{proof}
We only check the exactness of the sequence involving the covariant connecting morphism, the other case follows by duality. First observe that for every $h\in\hom{\A}{A}{X}$ we have, by definition, $\partial_{X}^{\eps}(h)=[\eps\cdot h]=\ext{\A}{h}{X}([\eps])=F(h,X)([\eps])\in F(C,X)$. Now the exactness of the sequences is consequence of \autoref{teo-ext-seq}\,$(b)$.  
\end{proof}
\end{comment}
We end this section by introducing the notion of proper functors. We will see in the next section how this definition fits well into the theory of subfunctors.

\begin{defi}\label{def-proper-subf}
We say that a functor $G\colon \C\to\sets$ is \textit{proper} provided $GX\neq\emptyset$ for all $X\in\C$. 
\end{defi}

Every additive functor from $\A$ to $\ab$ is clearly proper and the only non-proper functor from $\A$ to $\sets$ is the empty functor, that is, the functor which sends everything to empty, as we show next.

\begin{lema}\label{lema-funtor-propio}
If $\C$ is a preadditive category with zero object and $G\colon \C\to\sets$ is a functor such that $GX\neq \emptyset$ for some $X\in\C$, then $G$ is proper.
\end{lema}

\begin{proof}
For each $Y\in\A$ there is a well-defined correspondence $$G\colon \hom{\C}{X}{Y}\to\hom{\sets}{GX}{GY}.$$ Since $\C$ is preadditive with zero object, there exists a zero morphism $0_{X,Y}\colon X\to Y$ and so $G(0_{X,Y})\colon GX\to GY$ defines a morphism in $\sets$. Now from the hypothesis $GX\neq\emptyset$ and the well-known fact that there are no morphisms $GX\to\emptyset$ in $\sets$, necessarily $GY\neq\emptyset$.
\end{proof}

For our purposes, we list below equivalent conditions for a subfunctor to be proper.

\begin{prop}\label{prop-carac-subf-propio}
The following statements are equivalent for a subfunctor $F$: 
\begin{itemize}
\item[$(a)$] $F$ is proper.
\item[$(b)$] $\E_F$ is non-empty.
\item[$(c)$] $\E_0\subseteq \E_F$.
\item[$(d)$] There exists $[\eps]\in F(C,A)$ for some $A,C\in\A$.
\end{itemize}
\end{prop}

\begin{proof} 
According to \autoref{prop-biyec-subf-clas}, if $F$ is proper then $\E_F$ is non-empty and closed under pushouts and pullbacks, and therefore $\E_0\subseteq \E_F$ by \autoref{lema-eps0-eps}. Hence ($a$) $\Rightarrow$ ($b$) $\Rightarrow$ ($c$). Clearly ($c$) $\Rightarrow$ ($d$), and ($d$) $\Rightarrow$ ($a$) follows from \autoref{lema-funtor-propio}.
\end{proof}

\section{Closed subfunctors and h.f.\,classes}\label{section-3}

\subsection{f.\,classes and h.f.\,classes of morphisms} The axioms defining f.\,classes and h.f.\,classes were introduced and studied by Buchsbaum \citep{buchsbaum_note_1959, buchsbaum_satellites_1960} in order to develop a theory of relative homological algebra. Later on, Butler and Horrocks \citep{butler_classes_1961} formulated an equivalent set of axioms as follows.

\begin{defi}\label{def-fclass-hfclass}
Let $\mathcal{M}\subseteq \operatorname{Mor}(\A)$ be a class of morphisms in $\A$. Consider the following properties for $\mathcal{M}$:
\begin{itemize}
\item[(A)] $\mathcal{M}$ contains all zero monomorphisms and epimorphisms in $\A$.
\item[(B)] If $f\in \mathcal{M}$ and $f=xgy$ for some isomorphisms $x$ and $y$, then $g\in\mathcal{M}$.
\item[(C)] $f\in\mathcal{M}$ if, and only if, $k_f,c_f\in\mathcal{M}$, where $k_f$ and $c_f$ are the kernel and cokernel of $f$, respectively.
\item[(D)] If $f$ and $gf$ are monomorphisms and $gf\in\mathcal{M}$, then $f\in\mathcal{M}$.
\item[(D$^{\ast}$)] If $g$ and $gf$ are epimorphisms and $gf\in\mathcal{M}$, then $g\in\mathcal{M}$.
\item[(E)] If $f,g\in\mathcal{M}$ are monomorphisms and $gf$ is defined, then $gf\in\mathcal{M}$.
\item[(E$^{\ast}$)] If $f,g\in\mathcal{M}$ are epimorphisms and $gf$ is defined, then $gf\in\mathcal{M}$.
\end{itemize}
We say that $\mathcal{M}$ is an \textit{f.\,class} if satisfies properties (A)--(D$^{\ast}$). If $\mathcal{M}$ satisfies all of them then we call it an \textit{h.f.\,class}.
\end{defi}

\begin{obs}\label{obs-dual-fclass}
Notice that properties (A)--(C) are self-dual, while (D)-(D$^{\ast}$) and (E)-(E$^{\ast}$) are dual of each other. Therefore, $\mathcal{M}$ is an f.\,class in $\A$ if, and only if, $\mathcal{M}^{\op}$ is an f.\,class in $\A^{\op}$, where $\mathcal{M}^{\op}\coloneqq\{f^{\op}\colon f\in\mathcal{M}\}$. Also, $\mathcal{M}$ satisfies (E) if, and only if, $\mathcal{M}^{\op}$ satisfies (E$^{\ast}$), and thus $\mathcal{M}$ is an h.f.\,class if, and only if, $\mathcal{M}^{\op}$ is an h.f.\,class.
\end{obs}

\begin{lema}\label{lema-morf-sucex}
Let $\mathcal{M}\subseteq \operatorname{Mor}(\A)$ be a class of morphisms and $$0\longrightarrow A\overset{f}{\longrightarrow} B\overset{g}{\longrightarrow} C\longrightarrow 0$$ be an exact sequence in $\A$. If $\mathcal{M}$ satisfies properties \textup{(B)} and \textup{(C)} of \autoref{def-fclass-hfclass}, and $f\in\mathcal{M}$, then $g\in\mathcal{M}$. The dual statement also holds, that is, $f\in\mathcal{M}$ if $g\in\mathcal{M}$.\hfill $\qed$
\end{lema}

To any class of morphisms $\mathcal{M}\subseteq\operatorname{Mor}(\A)$ we can associate a collection of short exact sequences $\E_{\mathcal{M}}\subseteq \E$ whose elements are given by $$\eps\colon \ \ \ 0\longrightarrow A\overset{f}{\longrightarrow} B\overset{g}{\longrightarrow} C\longrightarrow 0\,,\quad \text{with}\quad f,g\in\mathcal{M}.$$ We will see that if $\mathcal{M}$ is an f.\,class, then $\E_{\mathcal{M}}$ induces a proper subfunctor, and reciprocally, every proper subfunctor gives rise to an f.\,class of morphisms. This is contained in the work of Butler and Horrocks \citep[Proposition 1.1, Proposition 1.2]{butler_classes_1961}.

\begin{prop}\label{prop-fclase-subf}
If $\mathcal{M}\subseteq \operatorname{Mor}(\A)$ is an f.\,class, then $F_{\mathcal{M}}\coloneqq F_{\E_{\mathcal{M}}}$ is a proper subfunctor.
\end{prop}

\begin{proof} According to \autoref{prop-biyec-subf-clas}, for $F_{\mathcal{M}}$ to be a subfunctor it is enough to see that the collection $\E_{\mathcal{M}}$ is closed under pullbacks and pushouts. For this, let $\eps\in\E_{\mathcal{M}}$ be such that $[\eps]\in \E_{\mathcal{M}}(A,B)$, and let $f\in\hom{\A}{X}{A}$. Let us consider the commutative diagram with exact rows arising from the pullback of $\eps$ along $f$, where by definition $i,d\in\mathcal{M}$: $$\begin{tikzcd}
	{\eps\cdot f\colon\ \ \ 0} & B & Z & X & 0 \\
	{\phantom{{}\cdot f}\eps\colon\ \ \ 0} & B & W & A & 0
	\arrow[from=1-1, to=1-2]
	\arrow["j", from=1-2, to=1-3]
	\arrow["z", from=1-3, to=1-4]
	\arrow[from=1-4, to=1-5]
	\arrow[from=2-1, to=2-2]
	\arrow["d"', from=2-3, to=2-4]
	\arrow[from=2-4, to=2-5]
	\arrow[Rightarrow, no head, from=1-2, to=2-2]
	\arrow["h"', from=1-3, to=2-3]
	\arrow["f", from=1-4, to=2-4]
	\arrow["i"', from=2-2, to=2-3]
	\arrow["\textup{\textsc{pb}}"{anchor=center}, draw=none, from=1-3, to=2-4]
\end{tikzcd}$$ Since $\mathcal{M}$ is an f.\,class and both $hj=i\in\mathcal{M}$ and $j$ are monomorphisms, then $j\in\mathcal{M}$ by (D) and so $\eps\cdot f \in \E_{\mathcal{M}}$ by \autoref{lema-morf-sucex}. Hence $\E_{\mathcal{M}}$ is closed under pullbacks. By arguing in a similar way or by duality, we see that $\E_{\mathcal{M}}$ is closed under pushouts.

Finally, from \autoref{prop-carac-subf-propio} we deduce that $F_{\mathcal{M}}$ is proper because being $\mathcal{M}$ an f.\,class, property (A) says that $\E_{\mathcal{M}}$ is non-empty since contains all short exact sequences of the form $$0\longrightarrow A\overset{1_A}\longrightarrow A\longrightarrow 0\longrightarrow 0\,.$$
\end{proof}

Next we recall the construction for obtaining a class of morphisms from a subfunctor. Given a subfunctor $F$, the class $\mathcal{M}_F\subseteq \operatorname{Mor}(\A)$ of \textit{$F$-morphisms} consists of those morphisms $f$ satisfying one of the following conditions:
\begin{itemize}
\item[(M1)] $f$ is monic and there exists $\eps\in\E_{F}$ such that  $$\eps\colon\ \ \ 0\longrightarrow A\overset{f}{\longrightarrow}B\longrightarrow C\longrightarrow 0.$$
\item[(M2)] $f$ is epic and there exists $\eps\in\E_{F}$ such that $$\eps\colon\ \ \  0\longrightarrow A\longrightarrow B\overset{f}{\longrightarrow} C\longrightarrow 0.$$
\item[(M3)] There exists $\eps_1,\eps_2\in\E_F$ such that \begin{align*}
&\eps_1\colon \ \ \  0\longrightarrow \operatorname{Ker}(f)\overset{k_f}{\longrightarrow} A\overset{g}{\longrightarrow} B\longrightarrow 0\,, \\ 
&\eps_2\colon\ \ \ 0\longrightarrow A\overset{h}{\longrightarrow}B\overset{c_f}{\longrightarrow}  \operatorname{Coker}(f)\longrightarrow 0\,.
\end{align*} 
\end{itemize}

As we mentioned before, we will show that the previous construction gives rise to an f.\,class of morphisms whenever the subfunctor is proper. To achieve this the next result will be helpful. 

\begin{lema}\label{lema-isos-ker.coker}
For an additive category $\C$ the following conditions hold:
\begin{itemize}
\item[$(a)$] If the kernel $k_f\colon \Ker{f}\to A$ of a morphism $f\colon A\to B$ exists in $\C$ and $y\colon A\to Y$ is an isomorphism, then the kernel of $fy^{-1}\colon Y\to B$ exists and is given by $yk_f\colon \Ker{f}\to Y$. Hence we have a commutative diagram $$\begin{tikzcd}
	{\operatorname{Ker}(f)} & A & B\, \\
	{\operatorname{Ker}(f)} & Y & B.
	\arrow[Rightarrow, no head, from=1-3, to=2-3]
	\arrow["{k_f}", from=1-1, to=1-2]
	\arrow["f", from=1-2, to=1-3]
	\arrow["{k_{fy^{-1}}}"', from=2-1, to=2-2]
	\arrow["{fy^{-1}}"', from=2-2, to=2-3]
	\arrow[Rightarrow, no head, from=1-1, to=2-1]
	\arrow["y", from=1-2, to=2-2]
\end{tikzcd}$$
\item[$(b)$] If the cokernel $c_f\colon B\to \Coker{f}$ of a morphism $f\colon A\to B$ exists in $\C$ and $x\colon X \to B$ is an isomorphism, then the cokernel of $x^{-1}f\colon A\to X$ exists and is given by $c_fx\colon X\to \Coker{f}$. Hence we have a commutative diagram $$\begin{tikzcd}
	A & X & {\operatorname{Coker}(f)}\,  \\
	A & B & {\operatorname{Coker}(f)}. 
	\arrow["{c_{x^{-1}f}}", from=1-2, to=1-3]
	\arrow["{c_f}"', from=2-2, to=2-3]
	\arrow["x", from=1-2, to=2-2]
	\arrow[Rightarrow, no head, from=1-3, to=2-3]
	\arrow["{x^{-1}f}", from=1-1, to=1-2]
	\arrow[Rightarrow, no head, from=1-1, to=2-1]
	\arrow["f"', from=2-1, to=2-2]
\end{tikzcd}$$
\end{itemize}\hfill \qed
\end{lema}

\begin{prop}\label{prop-subf-fclase}
If $F$ is a proper subfunctor, then the class $\mathcal{M}_{F}$ of $F$-morphisms is an f.\,class.
\end{prop}

\begin{proof}
We must check that $\mathcal{M}_F$ satisfies properties (A)--(D$^{\ast}$). In order to avoid repeating arguments we recall from \autoref{prop-biyec-subf-clas} and \autoref{prop-carac-subf-propio} that the collection $\E_F$ of $F$-exact sequences is closed under pushouts, pullbacks and isomorphisms, and contains the class $\E_0$ of split short exact sequences.

(A)\ For every $C\in\C$, the exact sequences $$0\longrightarrow C\overset{1_C}{\longrightarrow} C \longrightarrow  0\longrightarrow 0\ \quad \text{and}\quad 0\longrightarrow 0\longrightarrow C\overset{1_C}{\longrightarrow} C \longrightarrow 0$$ both belong to $\eps_F$ since all split short exact sequences are $F$-exact. Hence $\mathcal{M}_F$ contains all zero monomorphisms and epimorphisms.

(B)\ Let $f=xgy\in\mathcal{M}_F$ with $x\colon X\to B$ and $y\colon A\to Y$ isomorphisms. Then one of the following cases applies:

(i)\ If there exists $\eps\in\E_F$ such that $f$ appears as its monomorphism, then by taking the pushout of $\eps$ along $y$ we obtain the following commutative diagram with exact rows $$\begin{tikzcd}
	{\phantom{y\cdot}\eps\colon\ \ \ 0} & A & B & C & {0\,\phantom{.}} \\
	{y\cdot \eps\colon\ \ \ 0} & Y & X & C & {0\,.}
	\arrow[Rightarrow, no head, from=1-4, to=2-4]
	\arrow[from=1-1, to=1-2]
	\arrow["f", from=1-2, to=1-3]
	\arrow["h", from=1-3, to=1-4]
	\arrow[from=1-4, to=1-5]
	\arrow[from=2-1, to=2-2]
	\arrow["g"', from=2-2, to=2-3]
	\arrow["hx"', from=2-3, to=2-4]
	\arrow[from=2-4, to=2-5]
	\arrow["y"', from=1-2, to=2-2]
	\arrow["{x^{-1}}", from=1-3, to=2-3]
	\arrow["\textsc{po}"{anchor=center}, draw=none, from=1-2, to=2-3]
\end{tikzcd}$$ Thus $y\cdot\eps\in\E_F$ and according to (M1) in the definition of $\mathcal{M}_F$, this means that $g\in\mathcal{M}_F$. 

(ii)\ A similar reasoning as the previous one shows that $g\in\mathcal{M}_F$ in case there is some exact sequence $\eps\in\E_F$ containing $f$ as its epimorphism. 

(iii)\ If there are exact sequences $\eps_1\in\E_F$ containing $k_f$ as its monomorphism and $\eps_2\in\E_F$ containing $c_f$ as its epimorphism, given that $x$ and $y$ are isomorphisms then by \autoref{lema-isos-ker.coker} we have the equalities \begin{align*}
&\Ker{g}=\Ker{x^{-1}fy^{-1}}=\Ker{fy^{-1}}= \Ker{f},\\
&\Coker{g}=\Coker{x^{-1}fy^{-1}}= \Coker{x^{-1}f}= \Coker{f}.
\end{align*}
Hence we can form the following commutative  diagrams where the rows are exact because the vertical arrows are all isomorphisms: $$\begin{tikzcd}
	{\phantom{\eta_2}\eps_1\colon \ \ \ 0} & {\Ker{f}} & A & H & 0 \\
	{\phantom{\eps_2} \eta_1\colon \ \ \ 0} & {\Ker{g}} & Y & H & 0 \\
	{\phantom{\eps_2} \eta_2\colon \ \ \ 0} & M & X & {\operatorname{Coker}(g)} & 0 \\
	{\phantom{\eta_2} \eps_2\colon\ \ \ 0} & M & B & {\operatorname{Coker}(f)} & 0
	\arrow["{c_g}", from=3-3, to=3-4]
	\arrow[from=3-4, to=3-5]
	\arrow["{c_f}"', from=4-3, to=4-4]
	\arrow[from=4-4, to=4-5]
	\arrow["x", from=3-3, to=4-3]
	\arrow[Rightarrow, no head, from=3-4, to=4-4]
	\arrow["{x^{-1}m}", from=3-2, to=3-3]
	\arrow[Rightarrow, no head, from=3-2, to=4-2]
	\arrow["m"', from=4-2, to=4-3]
	\arrow[from=3-1, to=3-2]
	\arrow[from=4-1, to=4-2]
	\arrow[from=1-1, to=1-2]
	\arrow["{k_f}", from=1-2, to=1-3]
	\arrow["h", from=1-3, to=1-4]
	\arrow[from=1-4, to=1-5]
	\arrow[from=2-1, to=2-2]
	\arrow["{k_g}"', from=2-2, to=2-3]
	\arrow["{hy^{-1}}"', from=2-3, to=2-4]
	\arrow[from=2-4, to=2-5]
	\arrow[Rightarrow, no head, from=1-4, to=2-4]
	\arrow["y", from=1-3, to=2-3]
	\arrow[Rightarrow, no head, from=1-2, to=2-2]
\end{tikzcd}$$ Thus $\eta_1,\eta_2\in\E_F$ and according to (M3) in the definition of $\mathcal{M}_F$, $g\in\mathcal{M}_F$.

(C)\ Let $f\colon A\to B$ be a morphism in $\A$. If $k_f,c_f\in\mathcal{M}_F$, then there are two $F$-exact sequences $\eps_1$ and $\eps_2$, the first containing $k_f$ as its monomorphism and the latter containing $c_f$ as its epimorphism, which by (M3) of the definition of $\mathcal{M}_F$ says that $f\in\mathcal{M}_F$. 

Reciprocally, if $f\in\mathcal{M}_F$ then one of the following cases hold:

(i)\ If $f$ is the monomorphism of some $F$-exact sequence $\eps\in\E_F$, then $k_f=0\in\mathcal{M}_F$ by (A) and the universal property of the cokernel yields an isomorphism of exact sequences $$\begin{tikzcd}
	{\phantom{\xi}\eta\colon\ \ \ 0} & A & B & {\Coker{f}} & 0\,\phantom{.} \\
	{\phantom{\eta}\eps\colon\ \ \ 0} & A & B & C & 0\,.
	\arrow[from=1-1, to=1-2]
	\arrow["f", from=1-2, to=1-3]
	\arrow["{c_f}", from=1-3, to=1-4]
	\arrow[from=1-4, to=1-5]
	\arrow[from=2-1, to=2-2]
	\arrow["f"', from=2-2, to=2-3]
	\arrow["b"', from=2-3, to=2-4]
	\arrow[from=2-4, to=2-5]
	\arrow["{\exists !t}", dashed, from=1-4, to=2-4]
	\arrow[Rightarrow, no head, from=1-3, to=2-3]
	\arrow[Rightarrow, no head, from=1-2, to=2-2]
\end{tikzcd}$$ Hence $\eta\in\mathcal{M}_F$ and by (M2), $c_f\in\mathcal{M}_F$.

(ii)\ In a dual manner as in the previous case, we see that $k_f,c_f\in\mathcal{M}_F$ if $f$ is the epimorphism of some $F$-exact sequence.

(iii)\ If there exists $\eps_1,\eps_2\in\E_F$ such that $k_g$ is the monomorphism of $\eps_1$ and $c_f$ is the epimorphism of $\eps_2$, then $k_f\in\mathcal{M}_F$ by case (i) and $c_f\in\mathcal{M}_F$ by case (ii).

(D)\ If $f\colon A\to B$ and $g\colon A\to B$ are morphisms in $\A$ such that $gf\in\mathcal{M}_F$ and $f$ are both monomorphisms, then by the universal property of the cokernel there is a commutative diagram with exact rows and $\eps\in\E_F$: $$\begin{tikzcd}
	{\phantom{\eps} \eta\colon\ \ \ 0} & A & B & {\Coker{f}} & 0 \\
	{\phantom{\eta} \eps\colon\ \ \ 0} & A & B & C & 0
	\arrow[from=1-1, to=1-2]
	\arrow["f", from=1-2, to=1-3]
	\arrow["{c_f}", from=1-3, to=1-4]
	\arrow[from=1-4, to=1-5]
	\arrow[from=2-1, to=2-2]
	\arrow["gf"', from=2-2, to=2-3]
	\arrow["{c_{gf}}"', from=2-3, to=2-4]
	\arrow[from=2-4, to=2-5]
	\arrow["{\exists !w}", dashed, from=1-4, to=2-4]
	\arrow["g", from=1-3, to=2-3]
	\arrow[Rightarrow, no head, from=1-2, to=2-2]
\end{tikzcd}$$ Since $\eta$ is a pullback of $\eps$, then $\eta\in\E_F$ and so, by (M2), $f\in\mathcal{M}_F$.

(D$^{\ast}$)\ Follows in a similar way as we proof (D).
\end{proof}

The results from \autoref{prop-fclase-subf} and \autoref{prop-subf-fclase} suggests that both constructions are mutually inverse. Actually, this is the case as we show next.

\begin{teo}\label{teo-caract-subfpropio-fclases}
There is a bijection between the following two classes: 
\begin{itemize}
\item[($a$)] Proper subfunctors $F$ of $\ext{\A}{-}{-}\colon \A^{\op}\times\A\to\sets$.
\item[($b$)] f.\,classes $\mathcal{M}\subseteq \operatorname{Mor}(\A)$.
\end{itemize} 
The bijection is given by $F\mapsto \mathcal{M}_F$ and its inverse is given by $\mathcal{M}\mapsto F_{\mathcal{M}}$.
\end{teo}

\begin{proof}
We first observe that the maps $F\mapsto \mathcal{M}_F$ and $\mathcal{M}\mapsto F_{\mathcal{M}}$ are well-defined by \autoref{prop-subf-fclase} and \autoref{prop-fclase-subf}, respectively. We will show that $F_{\mathcal{M}_F}=F$ and $\mathcal{M}_{F_{\mathcal{M}}}=\mathcal{M}$, for all proper subfunctors $F$ and all f.\,classes of morphisms $\mathcal{M}$. 

From the bijection of \autoref{prop-biyec-subf-clas}, the equality $F_{\mathcal{M}_F}=F$ holds once we show $\E_{F_{\mathcal{M}_F}}=\E_F$. Since by definition $\E_{F_{\mathcal{M}_F}}=\E_{\mathcal{M}_F}$, we only need to see that $\E_{\mathcal{M}_F}=\E_F$. 

Consider $\eps\in\E_{\mathcal{M}_F}$ of the form $$\eps\colon \ \ \ 0\longrightarrow A\overset{f}{\longrightarrow} B \overset{g}{\longrightarrow} C \longrightarrow 0\,,\quad \text{with}\quad f,g\in\mathcal{M}_F.$$ Since $f\in\mathcal{M}_F$ is monic, by definition there is some $\eta\in\E_F$ such that $$\eta\colon\ \ \ 0\longrightarrow A\overset{f}{\longrightarrow} B \overset{h}{\longrightarrow} X \longrightarrow 0\,.$$ Hence there exists an isomorphism of exact sequences $$\begin{tikzcd}
	{\phantom{\eta}\eps\colon\ \ \ 0} & A & B & C & 0\,\phantom{.} \\
	{\phantom{\eps}\eta\colon\ \ \ 0} & A & B & X & 0\,.
	\arrow[from=1-1, to=1-2]
	\arrow["f", from=1-2, to=1-3]
	\arrow["g", from=1-3, to=1-4]
	\arrow[from=1-4, to=1-5]
	\arrow[from=2-1, to=2-2]
	\arrow["f"', from=2-2, to=2-3]
	\arrow["h"', from=2-3, to=2-4]
	\arrow[from=2-4, to=2-5]
	\arrow[Rightarrow, no head, from=1-2, to=2-2]
	\arrow[Rightarrow, no head, from=1-3, to=2-3]
	\arrow["{\exists !t}", dashed, from=1-4, to=2-4]
\end{tikzcd}$$ Thus $\eps\in\E_F$. This shows that $\E_{\mathcal{M}_F}\subseteq \E_F$. The other inclusion $\E_F\subseteq\E_{\mathcal{M}_F}$ is straightforward given that morphisms appearing in an $F$-exact sequence are by definition in the induced f.\,class $\mathcal{M}_F$.

Next we show $\mathcal{M}_{F_{\mathcal{M}}}=\mathcal{M}$. If $f\in\mathcal{M}_{F_{\mathcal{M}}}$ then one of the following cases applies:	
\begin{itemize}
\item[(i)] $f$ is either a monomorphism or an epimorphism and there exists $\eps\in\E_{F_{\mathcal{M}}}=\E_{\mathcal{M}}$ containing $f$. This means, according to the definition of $\E_{\mathcal{M}}$, that $f\in\mathcal{M}$.
\item[(ii)] There are $\eps_1,\eps_2\in\E_{F_{\mathcal{M}}}=\E_{\mathcal{M}}$ such that $\eps_1$ contains $k_f$ and $\eps_2$ contains $c_f$. As in case (i), $k_f,c_f\in\mathcal{M}$ and since $\mathcal{M}$ is an f.\,class, $f\in\mathcal{M}$ by (C).
\end{itemize} In any of the previous situations we conclude that $f\in\mathcal{M}$. Hence $\mathcal{M}_{F_{\mathcal{M}}}\subseteq \mathcal{M}$.

Now, if $f\in\mathcal{M}$, since $k_f,c_f\in\mathcal{M}$, then one of the following cases applies:
\begin{itemize}
\item[(i)] If $f$ is monic the sequence $$0\longrightarrow A\overset{f}{\longrightarrow} B\overset{c_f}{\longrightarrow} 	\Coker{f}\longrightarrow 0$$ belongs to $\E_{\mathcal{M}}=\E_{F_{\mathcal{M}}}$ and therefore $f\in\mathcal{M}_{F_{\mathcal{M}}}$.
\item[(ii)] If $f$ is epic, the sequence $$0\longrightarrow \Ker{f} \overset{k_f}{\longrightarrow} 	A\overset{f}{\longrightarrow} B\longrightarrow 0$$ belongs to $\E_{\mathcal{M}}=\E_{F_{\mathcal{M}}}$ and therefore $f\in\mathcal{M}_{F_{\mathcal{M}}}$.
\item[(iii)] If $f$ is not monic nor epic, then $k_f\in\mathcal{M}_{F_{\mathcal{M}}}$ by (i) and $c_f\in\mathcal{M}_{F_{\mathcal{M}}}$ by (ii), and since $\mathcal{M}_{F_{\mathcal{M}}}$ is an f.\,class then $f\in\mathcal{M}_{F_{\mathcal{M}}}$.
\end{itemize} In any of the previous situations we conclude that $f\in\mathcal{M}_{}$. Hence $\mathcal{M}\subseteq\mathcal{M}_{F_{\mathcal{M}}} $. 
\end{proof}

From the previous result together with \autoref{prop-biyec-subf-clas} and \autoref{prop-carac-subf-propio} we immediately deduce the following.

\begin{coro}
There is a bijection between any two of the following classes:
\begin{itemize}
\item[$(a)$] Proper subfunctors $F$ of $\ext{\A}{-}{-}\colon\A^{\op}\times\A\to\sets$.
\item[$(b)$] f.\,classes $\mathcal{M}\subseteq\operatorname{Mor}(\A)$.
\item[$(c)$] Non-empty collections $\E\subseteq \E_{\A}$ closed under pushouts and pullbacks. 
\end{itemize} \hfill $\qed$

\end{coro}

We end this subsection by giving conditions for a proper subfunctor to be additive. This is stated in \citep[Theorem 1.1]{butler_classes_1961}.

\begin{prop}\label{prop-E1oE2-aditivo}
Let $F$ be a proper subfunctor and $\mathcal{M}_F$ be the f.\,class induced by $F$ (see \autoref{teo-caract-subfpropio-fclases}). If $\mathcal{M}_F$ satisfies either \textup{(E)} or \textup{(E$^{\ast}$)}, then $F$ is an additive subfunctor.
\end{prop}

\begin{proof}
In view of \autoref{prop-biyec-subf-clas} we will show that $\E_F$ is closed under finite direct sums provided $\mathcal{M}_F$ satisfies (E). The case when $\mathcal{M}_F$ satisfies (E$^{\ast}$) is deal in a similar manner or by duality. 

For $k=1,2$, let $\eps_k\in\E_F(C,A)$ be given by $$\eps_k\colon\ \ \ 0\longrightarrow A\overset{f_k}{\longrightarrow} B_k\overset{g_k}{\longrightarrow} C\longrightarrow 0\, .$$ We will see that $\eps_1\oplus \eps_2\in\E_F$, or equivalently, that $f_1\oplus f_2\in\mathcal{M}_F$. Indeed, first notice that the monomorphism $f_1\oplus f_2\colon A\oplus A\to B_1\oplus B_2$ factors through a pair of monomorphisms as  $$f_1\oplus f_2=\begin{pmatrix}
f_1 & 0 \\ 0 & f_2
\end{pmatrix}=\begin{pmatrix}
f_1 & 0 \\ 0 & 1_{B_2}
\end{pmatrix} \begin{pmatrix}
1_{B_1} & 0 \\ 0 & f_2
\end{pmatrix}.$$ So the proof will be complete if we are able to see that each one of the monomorphisms on the right side is in $\mathcal{M}_F$, for $\mathcal{M}_F$ satisfying (E) will lead us to $f_1\oplus f_2\in\mathcal{M}_F$. From the commutativity of the diagrams with exact rows displayed below 
$$\begin{tikzcd}[ampersand replacement=\&, column sep = 3em]
	\phantom{\tau_1}\eps_1\colon \ \ \ 0 \& A \& {B_1} \& C \& 0\,\phantom{,} \\
	\phantom{\eps_1}\tau_1\colon \ \ \ 0 \& {A\oplus B_2} \& {B_1\oplus B_2} \& C \& 0\, ,
	\arrow[from=1-1, to=1-2]
	\arrow["{f_1}", from=1-2, to=1-3]
	\arrow["{\begin{psmallmatrix}1 \\ 0\end{psmallmatrix}}"', from=1-2, to=2-2]
	\arrow["{g_1}", from=1-3, to=1-4]
	\arrow["{\begin{psmallmatrix}1 \\ 0\end{psmallmatrix}}", from=1-3, to=2-3]
	\arrow[from=1-4, to=1-5]
	\arrow[from=2-1, to=2-2]
	\arrow["{\begin{psmallmatrix}f_1 & 0 \\ 0 & 1\end{psmallmatrix}}"', from=2-2, to=2-3]
	\arrow["{\begin{psmallmatrix}g_1 & 0\end{psmallmatrix}}"', from=2-3, to=2-4]
	\arrow[equals, from=2-4, to=1-4]
	\arrow[from=2-4, to=2-5]
\end{tikzcd}$$

$$\begin{tikzcd}[ampersand replacement=\&, column sep = 3em]
	\phantom{\tau_2}\eps_2\colon \ \ \ 0 \& A \& {B_2} \& C \& 0\,\phantom{,} \\
	\phantom{\eps_2}\tau_2\colon \ \ \ 0 \& {B_1\oplus A} \& {B_1\oplus B_2} \& C \& 0\, ,
	\arrow[from=1-1, to=1-2]
	\arrow["{f_2}", from=1-2, to=1-3]
	\arrow["{\begin{psmallmatrix}0 \\ 1\end{psmallmatrix}}"', from=1-2, to=2-2]
	\arrow["{g_2}", from=1-3, to=1-4]
	\arrow["{\begin{psmallmatrix}0 \\ 1\end{psmallmatrix}}", from=1-3, to=2-3]
	\arrow[from=1-4, to=1-5]
	\arrow[from=2-1, to=2-2]
	\arrow["{\begin{psmallmatrix}1 & 0 \\ 0 & f_2\end{psmallmatrix}}"', from=2-2, to=2-3]
	\arrow["{\begin{psmallmatrix}0 & g_2\end{psmallmatrix}}"', from=2-3, to=2-4]
	\arrow[equals, from=2-4, to=1-4]
	\arrow[from=2-4, to=2-5]
\end{tikzcd}$$ we deduce that $\tau_1,\tau_2\in\mathcal{E}_F$ as they are pushouts of $\eps_1$ and $\eps_2$, respectively, and $\mathcal{E}_F$ is closed under pushouts. Therefore, $\begin{psmallmatrix}f_1 & 0 \\ 0 & 1_{B_2}\end{psmallmatrix},\begin{psmallmatrix}1_{B_1} & 0 \\ 0 & f_2\end{psmallmatrix}\in\mathcal{M}_F$.
\end{proof}

We will see in \autoref{teo-caract-cerr-hf} that the converse of \autoref{prop-E1oE2-aditivo} holds true.

\subsection{Closed subfunctors} 
The notion of closed subfunctor was very important for the development of relative theories. We will see next the relation between closed subfunctors and h.f.\,classes introduced previously.
\begin{defi}
Let $F\colon \A^{\op}\times \A\to\ab$ be a proper subfunctor. We say that:
\begin{itemize}
\item[(a)] $F$ is \textit{closed on the right} if for each $C\in \A$, the induced functor $F(C,-)\colon \A\to \ab$ is half-exact over the class of $F$-exact sequences, that is, if given an $F$-exact sequence $0\to X\to Y\to Z\to 0$, then $F(C,X)\to F(C,Y)\to F(C,Z)$ is exact in $\ab$.
\item[(b)] $F$ is \textit{closed on the left} if for each $A\in \A$, the induced functor $F(-,A)\colon \A^{\op}\to \ab$ is half-exact over the class of $F$-exact sequences, that is, if given an $F$-exact sequence $0\to X\to Y\to Z\to 0$, then $F(Z,A)\to F(Y,A)\to F(X,A)$ is exact in $\ab$.
\item[(c)] $F$ is \textit{closed} if it is both closed on the right and on the left.
\end{itemize}
\end{defi}

Immediately from the definition we deduce that the closedness conditions only applies to additive subfunctors as we state below.

\begin{prop}
If $F$ is closed on the right (left, respectively) then the induced functors $F(X,-)$ ($F(-,X)$, respectively) are additive, for each $X\in\A$. In particular, every closed subfunctor is additive.
\end{prop}

\begin{proof}
This is essentially the proof that half-exact functors between abelian categories are additive. For this, one shows that half-exact functors preserve split short exact sequences and therefore also preserve finite direct sums (see for instance \citep[Section 4.6, Proposition 1]{pareigis_categories_1970}). By \autoref{lema-eps0-eps}, for any proper subfunctor $F$ we have that the split short exact sequences are $F$-exact and thus, by the above reasoning, each of the induced functors are additive provided $F$ is closed.
\end{proof}

\begin{prop}
Let $F$ be a proper subfunctor and $$\eps\colon\ \ \ 0\longrightarrow A\overset{f}{\longrightarrow} B \overset{g}{\longrightarrow} C\longrightarrow 0$$ be an $F$-exact sequence. If $F$ is closed on the left, then for every $X\in\A$ the induced sequence of abelian groups $$\begin{tikzcd}[column sep = 1em]
	0 & {\hom{\A}{C}{X}} & {\hom{\A}{B}{X}} & {\hom{\A}{A}{X}} \\
	& {F(C,X)} & {F(B,X)} & {F(A,X)}
	\arrow[from=1-1, to=1-2]
	\arrow[from=1-2, to=1-3]
	\arrow[from=1-3, to=1-4]
	\arrow[from=2-3, to=2-4]
	\arrow[from=2-2, to=2-3]
	\arrow["{\delta_{X}^{\eps}}"', from=1-4, to=2-2,out=-8, in=172]{dll}
\end{tikzcd}$$ is exact. Dually, if $F$ is closed on the right, then the induced sequence of abelian groups $$\begin{tikzcd}[column sep = 1em]
	0 & {\hom{\A}{X}{A}} & {\hom{\A}{X}{B}} & {\hom{\A}{X}{C}} \\
	& {F(X,A)} & {F(X,B)} & {F(X,C)}
	\arrow[from=1-1, to=1-2]
	\arrow[from=1-2, to=1-3]
	\arrow[from=1-3, to=1-4]
	\arrow[from=2-3, to=2-4]
	\arrow[from=2-2, to=2-3]
	\arrow["{\partial_{X}^{\eps}}"', from=1-4, to=2-2,out=-8, in=172]{dll}
\end{tikzcd}$$ is exact.
\end{prop}

\begin{proof}
We only prove the case when $F$ is closed on the left since the case when is closed on the right follows by duality. Indeed, from \autoref{prop-subf-seq} and the fact that $F(-,X)$ is half-exact over $\E_F$, the only thing we have to show is the exactness of the sequence $$\hom{\A}{A}{X} \overset{\delta_{X}^{\eps}}{\longrightarrow} F(C,X)\overset{F(g,X)}{\longrightarrow} F(B,X).$$ Indeed, given that $\eps\in\E_F$ and $\E_F$ is closed under pullbacks, by \autoref{prop-subf-seq} the images of the contravariant connecting morphism $\delta^{\eps}_{X}$ lies in $F(C,X)$ and therefore, from the exactness of the sequence in \autoref{teo-ext-seq}\,$(b)$, we get the relations $$\operatorname{Im}(\delta^{\eps}_{X})\subseteq \operatorname{Ker}(F(g,X))\subseteq \operatorname{Ker}(\ext{\A}{g}{X})= \operatorname{Im}(\delta^{\eps}_{X}).$$
\end{proof}

We are now in position to state the promised converse of \autoref{prop-E1oE2-aditivo} which is contained in \citep[Theorem 1.1]{butler_classes_1961}.

\begin{teo}\label{teo-caract-cerr-hf}
For a proper subfunctor $F$ the following holds true:
\begin{itemize}
\item[$(a)$] $F$ is closed on the right if, and only if, $\mathcal{M}_F$ satisfies \textup{(E)}.
\item[$(b)$] $F$ is closed on the left if, and only if, $\mathcal{M}_F$ satisfies \textup{(E$^{\ast}$)}.
\end{itemize}
Therefore $F$ is closed if, and only if, $\mathcal{M}_F$ is an h.f.\,class.
\end{teo}

\begin{proof}
We only prove ($a$) because ($b$) follows by duality.

($\Leftarrow$)\ We show that $F$ is closed on the right provided $\mathcal{M}_F$ satisfies (E). Given an $F$-exact sequence $$0\longrightarrow A\overset{f}{\longrightarrow} B \overset{g}{\longrightarrow} C\longrightarrow 0\,,$$ we have to see that the induced sequence $$F(X,A)\overset{F(X,f)}{\longrightarrow} F(X,B)\overset{F(X,g)}{\longrightarrow} F(X,C)$$ is exact in $\ab$, for all $X\in\A$. By \autoref{prop-E1oE2-aditivo}, $F$ is additive and so we only have to proof $\Ker{F(X,g)}\subseteq \Img{F(X,f)}$.

Let $[\eta]\in\Ker{F(X,g)}$. Since $\Ker{F(X,g)}\subseteq \Ker{\ext{\A}{X}{g}}=\Img{\ext{\A}{X}{f}}$ by \autoref{teo-ext-seq}, then $[\eta]=[f\cdot \lambda]$ for some $\lambda\in\E_{\A}$ and hence there is a commutative diagram with exact rows $$\begin{tikzcd}
	{\phantom{\eta}\lambda\colon\ \ \ 0} & A & Z & X & {0\,\phantom{.}} \\
	{\phantom{\lambda}\eta\colon\ \ \ 0} & B & W & X & {0\,.}
	\arrow[from=1-1, to=1-2]
	\arrow["h", from=1-2, to=1-3]
	\arrow[from=1-3, to=1-4]
	\arrow[from=1-4, to=1-5]
	\arrow["u"', from=2-2, to=2-3]
	\arrow[from=2-3, to=2-4]
	\arrow[from=2-4, to=2-5]
	\arrow[Rightarrow, no head, from=1-4, to=2-4]
	\arrow["w", from=1-3, to=2-3]
	\arrow["f"', from=1-2, to=2-2]
	\arrow[from=2-1, to=2-2]
	\arrow["\textsc{po}"{anchor=center}, draw=none, from=1-2, to=2-3]
\end{tikzcd}$$ Given that $u,f\in\mathcal{M}_F$ are monomorphisms, by (E) we get $uf=wh\in\mathcal{M}_F$, and since $wh$ and $h$ are monomorphisms, also $h\in\mathcal{M}_F$ by (D). Hence $\lambda\in\E_{F_{\mathcal{M}_F}}$. But from \autoref{teo-caract-subfpropio-fclases}, $\E_{F_{\mathcal{M}_F}}=\E_F$ because $F$ is proper and so $[\eta]=F(X,f)([\lambda])\in\Img{F(X,f)}$.

($\Rightarrow$)\ We will show that $\mathcal{M}_F$ satisfies (E) provided $F$ is closed on the right. Let $f\colon A\to B$ and $g\colon B\to C$ be monomorphisms in $\mathcal{M}_F$. Let us show that $h\coloneqq gf\in\mathcal{M}_F$. By an application of the Snake Lemma there is a commutative diagram with exact rows and columns \begin{equation}\label{eq-diag-h=gf}
\begin{tikzcd}
	& 0 & 0 & 0 \\
	0 & A & A & 0 & 0 \\
	0 & B & C & {\Coker{g}} & 0 \\
	0 & {\Coker{f}} & {\Coker{h}} & {\Coker{g}} & 0 \\
	& 0 & 0 & 0
	\arrow[from=2-1, to=2-2]
	\arrow[Rightarrow, no head, from=2-2, to=2-3]
	\arrow[from=2-3, to=2-4]
	\arrow[from=2-4, to=2-5]
	\arrow["g"', from=3-2, to=3-3]
	\arrow["{c_g}"', from=3-3, to=3-4]
	\arrow[from=3-4, to=3-5]
	\arrow[from=2-4, to=3-4]
	\arrow["h", from=2-3, to=3-3]
	\arrow["f"', from=2-2, to=3-2]
	\arrow[from=3-1, to=3-2]
	\arrow[from=1-2, to=2-2]
	\arrow["{c_f}"', from=3-2, to=4-2]
	\arrow[from=4-2, to=5-2]
	\arrow[from=1-3, to=2-3]
	\arrow["{c_h}", from=3-3, to=4-3]
	\arrow[from=4-3, to=5-3]
	\arrow[from=1-4, to=2-4]
	\arrow[Rightarrow, no head, from=3-4, to=4-4]
	\arrow[from=4-1, to=4-2]
	\arrow["m"', from=4-2, to=4-3]
	\arrow["n"', from=4-3, to=4-4]
	\arrow[from=4-4, to=4-5]
	\arrow[from=4-4, to=5-4]
\end{tikzcd}
\end{equation}
Therefore, we deduce the following commutative diagram with exact rows: $$\begin{tikzcd}
	{\phantom{_{hg}} \eps_f\colon\ \ \ 0} & A & B & {\Coker{f}} & 0 \\
	{\phantom{_{fg}} \eps_h\colon\ \ \ 0} & A & C & {\Coker{h}} & 0 \\
	{\phantom{_{hf}} \eps_g\colon\ \ \ 0} & B & C & {\Coker{g}} & 0
	\arrow[from=2-1, to=2-2]
	\arrow["h"', from=2-2, to=2-3]
	\arrow["{c_h}"', from=2-3, to=2-4]
	\arrow[from=2-4, to=2-5]
	\arrow["g"', from=3-2, to=3-3]
	\arrow["{c_g}"', from=3-3, to=3-4]
	\arrow[from=3-4, to=3-5]
	\arrow["n", from=2-4, to=3-4]
	\arrow[Rightarrow, no head, from=2-3, to=3-3]
	\arrow["f"', from=2-2, to=3-2]
	\arrow[from=3-1, to=3-2]
	\arrow[Rightarrow, no head, from=1-2, to=2-2]
	\arrow[from=1-1, to=1-2]
	\arrow["f", from=1-2, to=1-3]
	\arrow["{c_f}", from=1-3, to=1-4]
	\arrow[from=1-4, to=1-5]
	\arrow["m", from=1-4, to=2-4]
	\arrow["g", from=1-3, to=2-3]
\end{tikzcd}$$ In particular, from \autoref{prop-propiedades-pullback-pushout}, we have $[f\cdot\eps_h]=[\eps_g\cdot n]$ and given that $\eps_g\in\E_F$ and $\E_F$ is closed under pullbacks, then $f\cdot\eps_h\in\E_F$. Moreover, $0=[0\cdot \eps_h]=[(c_f f)\cdot \eps_h]=[c_f\cdot(f\cdot\eps_h)]$, so $[f\cdot \eps_h]\in\Ker{F(\Coker{h},c_f}$. Since $\eps_f\in\E_F$ and $F$ is closed on the right, from the exactness of the induced sequence $$F(\Coker{h}, A)\longrightarrow F(\Coker{h}, B)\longrightarrow F(\Coker{h}, \Coker{f})\,,$$ we have $\Ker{F(\Coker{h},c_f)}=\Img{F(\Coker{h},f)}$. Hence there exists $\eta\in\E_F$ such that $[f\cdot\eps_h]=[f\cdot\eta]$ and thus $[\eps_h]-[\eta]\in \Ker{\ext{\A}{\Coker{h}}{f}}=\operatorname{Im}\left(\partial_{\Coker{h}}^{\eps_f}\right)$. That is, there exists $t\colon \Coker{h}\to \Coker{f}$ such that $[\eps_h]-[\eta]=[\eps_f\cdot t]$. Given that $\E_F$ is closed under pullbacks, this means that $\eps_f\cdot t\in\E_F$ and since $F(\Coker{h},A)\in\ab$, then $[\eps_h]=[\eps_f\cdot t]+[\eta]\in F(\Coker{h}, A)$. Therefore, $h\in\mathcal{M}_F$.
\end{proof}

From the previous result we see that closed subfunctors correspond to h.f.\,classes via the bijection of \autoref{teo-caract-subfpropio-fclases} between proper subfunctors and f.\,classes.

\begin{coro}
The bijection from \autoref{teo-caract-subfpropio-fclases} restricts to a bijection between the following:
\begin{itemize}
\item[($a$)] Closed subfunctors $F$ of $\ext{\A}{-}{-}\colon\A^{\op}\times\A\to\ab$.
\item[($b$)] h.f.\,classes $\mathcal{M}\subseteq \operatorname{Mor}(\A)$.
\end{itemize}\hfill $\qed$
\end{coro}

\subsection{The $3\times 3$-lemma property for subfunctors.}

In what follows, we prove the result of Buan \citep{buan_closed_2001} which states that closed subfunctors are precisely those satisfying the $3 \times 3$-lemma property. We then show how this can be used to easily deduce the closedness of certain subfunctors, using the $3 \times 3$-lemma for abelian categories.

\begin{defi}
We say that a proper subfunctor $F$ \textit{has the $3\times 3$-lemma property} if given a commutative diagram with exact rows and columns $$\begin{tikzcd}
	& 0 & 0 & 0 \\
	0 & {A} & {B} & {C} & 0 \\
	0 & {D} & {E} & {G} & 0 \\
	0 & {H} & {I} & {J} & 0 \\
	& 0 & 0 & 0
	\arrow[from=2-1, to=2-2]
	\arrow["a", from=2-2, to=2-3]
	\arrow["b", from=2-3, to=2-4]
	\arrow[from=2-4, to=2-5]
	\arrow[from=3-1, to=3-2]
	\arrow["d", from=3-2, to=3-3]
	\arrow["e", from=3-3, to=3-4]
	\arrow[from=3-4, to=3-5]
	\arrow[from=4-1, to=4-2]
	\arrow["g", from=4-2, to=4-3]
	\arrow["h", from=4-3, to=4-4]
	\arrow[from=4-4, to=4-5]
	\arrow[from=1-2, to=2-2]
	\arrow["i"', from=2-2, to=3-2]
	\arrow["k"', from=3-2, to=4-2]
	\arrow[from=4-2, to=5-2]
	\arrow[from=1-3, to=2-3]
	\arrow["j"', from=2-3, to=3-3]
	\arrow["l"', from=3-3, to=4-3]
	\arrow[from=4-3, to=5-3]
	\arrow[from=1-4, to=2-4]
	\arrow["c", from=2-4, to=3-4]
	\arrow["f", from=3-4, to=4-4]
	\arrow[from=4-4, to=5-4]
\end{tikzcd}$$ such that the first and third rows and all columns are $F$-exact, then also the second row is $F$-exact.
\end{defi}

The following preliminary result will be used in the proof of \autoref{teo-principal}.

\begin{lema}\label{lema-popb-3x3}
Let $\mathcal{M}\subseteq\operatorname{Mor}(\A)$ be a class of morphisms.
\begin{itemize}
\item[($a)$] Let $(f,g,1_C)\colon \eps\to \eta$ be a morphism, that is, a commutative diagram with exact rows $$\begin{tikzcd}
	{\phantom{\eta}\eps\colon\ \ \ 0} & A & B & C & 0\,\phantom{.} \\
	{\phantom{\eps}\eta\colon\ \ \ 0} & D & E & C & 0\,.
	\arrow[from=1-1, to=1-2]
	\arrow["a", from=1-2, to=1-3]
	\arrow[from=2-1, to=2-2]
	\arrow["d"', from=2-2, to=2-3]
	\arrow["e"', from=2-3, to=2-4]
	\arrow[from=2-4, to=2-5]
	\arrow["b", from=1-3, to=1-4]
	\arrow[from=1-4, to=1-5]
	\arrow[Rightarrow, no head, from=1-4, to=2-4]
	\arrow["g"', from=1-3, to=2-3]
	\arrow["f"', from=1-2, to=2-2]
\end{tikzcd}$$ If $\mathcal{M}$ is an f.class and $f\in\mathcal{M}$ is a monomorphism, then $g\in\mathcal{M}$.
\item[$(b)$] Let $(1_A,g,h)\colon \eps\to \eta$ be a morphism, that is, a commutative diagram with exact rows $$\begin{tikzcd}
	{\phantom{\eta}\eps\colon\ \ \ 0} & A & B & C & 0\,\phantom{.} \\
	{\phantom{\eps}\eta\colon\ \ \ 0} & A & D & E & 0\,.
	\arrow[from=1-1, to=1-2]
	\arrow["a", from=1-2, to=1-3]
	\arrow[from=2-1, to=2-2]
	\arrow["d"', from=2-2, to=2-3]
	\arrow["e"', from=2-3, to=2-4]
	\arrow[from=2-4, to=2-5]
	\arrow["b", from=1-3, to=1-4]
	\arrow[from=1-4, to=1-5]
	\arrow[Rightarrow, no head, from=1-2, to=2-2]
	\arrow["g", from=1-3, to=2-3]
	\arrow["h", from=1-4, to=2-4]
\end{tikzcd}$$ If $\mathcal{M}$ is an h.f.class, $e,h\in\mathcal{M}$ and $h$ is a monomorphism, then $g\in\mathcal{M}$.
\end{itemize}
\end{lema}

\begin{proof}
($a$)\ By an application of the Snake Lemma, there exists an isomorphism $t\colon \operatorname{Coker}(f)\to\operatorname{Coker}(g)$ such that $c_gd=tc_f$. Since $\mathcal{M}$ is an f.\,class and $f\in\mathcal{M}$ is a monomorphism, then $tc_f\in\mathcal{M}$ by (C) and (B), and thus $c_g\in\mathcal{M}$ by (D$^{\ast}$). Finally observe that $g$, being parallel to a monomorphism in a pushout diagram, is also a monomorphism and so, $g=k_{c_g}\in\mathcal{M}$ by (C).

($b$)\ By an application of the Snake Lemma, there exists an isomorphism $u\colon \operatorname{Coker}(g)\to\operatorname{Coker}(h)$ such that $uc_g=c_he$. Since $\mathcal{M}$ is an h.f.\,class and $e,c_h\in\mathcal{M}$ are both epimorphisms, then $c_he\in\mathcal{M}$ by (E$^{\ast}$) and thus $c_g=u^{-1}c_he\in\mathcal{M}$ by (B). Finally observe that $g$, being parallel to a monomorphism in a pullback diagram, is also a monomorphism and so, $g=k_{c_g}\in\mathcal{M}$ by (C).
\end{proof}

\begin{obs}\label{obs-dual-lema-popb-3x3}
In the previous lemma we could state and prove the dual ($b^{\ast}$) of ($b$) and then ($b$) will follow by duality arguing in $\A^{\op}$, see \autoref{obs-dual-fclass}. In that case, we will instead apply property (E) for proving this.
\end{obs}

\begin{teo}\label{teo-principal}
The following conditions are equivalent for a proper subfunctor $F$:
\begin{itemize}
\item[$(a)$] $F$ is closed.
\item[$(b)$] $\mathcal{M}_F$ is an h.f.\,class.
\item[$(c)$] $F$ has the $3\times 3$--lemma property.
\end{itemize}
\end{teo}

\begin{proof}
\autoref{teo-caract-cerr-hf} establishes the equivalence ($a$) $\Leftrightarrow$ ($b$). If $f,g\in\mathcal{M}_F$ are monomorphisms such that $h\coloneqq gf$ is defined, then as in the proof of \autoref{teo-caract-cerr-hf} we can form the commutative diagram (\ref{eq-diag-h=gf}), where all rows and the first and third columns are $F$-exact. Hence, if $F$ has the $3\times 3$--lemma property, then $h\in\mathcal{M}_F$ and so $\mathcal{M}_F$ satisfies (E). By a similar construction or by duality, we see that $\mathcal{M}_F$ satisfies (E$^{\ast}$). Since we already know from \autoref{prop-subf-fclase} that $\mathcal{M}_F$ is an f.\,class, the previous reasoning shows that ($c$) $\Rightarrow$ ($b$).

To see that ($b$) $\Rightarrow$ ($c$), let us suppose that $\mathcal{M}_F$ is an h.f.\,class and consider a commutative diagram \begin{equation}\label{eq-diag-3x3}
\begin{tikzcd}
	& 0 & 0 & 0 \\
	0 & {A} & {B} & {C} & 0 \\
	0 & {D} & {E} & {G} & 0 \\
	0 & {H} & {I} & {J} & 0 \\
	& 0 & 0 & 0
	\arrow[from=2-1, to=2-2]
	\arrow["a", from=2-2, to=2-3]
	\arrow["b", from=2-3, to=2-4]
	\arrow[from=2-4, to=2-5]
	\arrow[from=3-1, to=3-2]
	\arrow["d", from=3-2, to=3-3]
	\arrow["e", from=3-3, to=3-4]
	\arrow[from=3-4, to=3-5]
	\arrow[from=4-1, to=4-2]
	\arrow["g", from=4-2, to=4-3]
	\arrow["h", from=4-3, to=4-4]
	\arrow[from=4-4, to=4-5]
	\arrow[from=1-2, to=2-2]
	\arrow["i"', from=2-2, to=3-2]
	\arrow["k"', from=3-2, to=4-2]
	\arrow[from=4-2, to=5-2]
	\arrow[from=1-3, to=2-3]
	\arrow["j"', from=2-3, to=3-3]
	\arrow["l"', from=3-3, to=4-3]
	\arrow[from=4-3, to=5-3]
	\arrow[from=1-4, to=2-4]
	\arrow["c", from=2-4, to=3-4]
	\arrow["f", from=3-4, to=4-4]
	\arrow[from=4-4, to=5-4]
\end{tikzcd}\end{equation} such that the first and third row and all columns are $F$-exact. Let us denote by $\eps_1,\eps_2,\eps_3$ its rows and by $\eta_1,\eta_2,\eta_3$ its columns. We must show that $\eps_2$ is also an $F$-exact sequence. According to \autoref{lema-morf-sucex} it suffices to show that $d \in \mathcal{M}_F$ given that $\mathcal{M}_F$ is an h.f.\,class.

Since $(i,j,c)\colon \eps_1\to\eps_2$ is a morphism, then \autoref{prop-propiedades-pullback-pushout} provides a factorization $(i,j,c)=(1_D,o,c)(i,q,1_C)$ such that the following diagram commutes and has exact rows $$\begin{tikzcd}
	{\phantom{\varepsilon\varepsilon_2}\varepsilon_1\colon\ \ \ 0} & A & B & C & 0\,\phantom{.} \\
	{\phantom{\varepsilon_1\varepsilon_2}\varepsilon\colon\ \ \ 0} & D & M & C & 0\,\phantom{.} \\
	{\phantom{\varepsilon\varepsilon_1}\varepsilon_2\colon\ \ \ 0} & D & E & G & 0\,.
	\arrow[from=1-1, to=1-2]
	\arrow["a", from=1-2, to=1-3]
	\arrow["i"', from=1-2, to=2-2]
	\arrow["b", from=1-3, to=1-4]
	\arrow["q", from=1-3, to=2-3]
	\arrow[from=1-4, to=1-5]
	\arrow[equals, from=1-4, to=2-4]
	\arrow[from=2-1, to=2-2]
	\arrow["m", from=2-2, to=2-3]
	\arrow[equals, from=2-2, to=3-2]
	\arrow["n", from=2-3, to=2-4]
	\arrow["o", from=2-3, to=3-3]
	\arrow[from=2-4, to=2-5]
	\arrow["c", from=2-4, to=3-4]
	\arrow[from=3-1, to=3-2]
	\arrow["d"', from=3-2, to=3-3]
	\arrow["e"', from=3-3, to=3-4]
	\arrow[from=3-4, to=3-5]
\end{tikzcd}$$ Given that $i\in\mathcal{M}_F$ is a monomorphism, then $q\in\mathcal{M}_F$ by \autoref{lema-popb-3x3}\,($a$). Furthermore, an application of the Snake Lemma to the morphism $(i,q,1_C)\colon \eps_1\to \eps$ yields an isomorphism $t\colon \operatorname{Coker}(q)\to H$ such that the following diagram commutes
\begin{equation}\label{eq-diag-snake3x3}
\begin{tikzcd}
	{\eps_1\colon\ \ \ 0} & A & B & C & 0 \\
	{\phantom{{}_1}\eps\colon\ \ \ 0} & D & M & C & 0 \\
	& H & {\operatorname{Coker}(q)}
	\arrow[from=1-1, to=1-2]
	\arrow["a", from=1-2, to=1-3]
	\arrow[from=2-1, to=2-2]
	\arrow["m"', from=2-2, to=2-3]
	\arrow["n"', from=2-3, to=2-4]
	\arrow[from=2-4, to=2-5]
	\arrow["b", from=1-3, to=1-4]
	\arrow[from=1-4, to=1-5]
	\arrow["i"', from=1-2, to=2-2]
	\arrow["q", from=1-3, to=2-3]
	\arrow[Rightarrow, no head, from=1-4, to=2-4]
	\arrow["{c_q}", from=2-3, to=3-3]
	\arrow["k"', from=2-2, to=3-2]
	\arrow["t", from=3-3, to=3-2]
	\arrow["p", dashed, from=2-3, to=3-2]
\end{tikzcd}
\end{equation} By setting $p\coloneqq tc_q\colon M\to H$, we arrive at a commutative diagram with exact rows $$\begin{tikzcd}
	{\eta_1\colon\ \ \ 0} & A & D & H & 0\,\phantom{.} \\
	{\phantom{{}_1}\eta\colon\ \ \ 0} & B & M & H & 0\,.
	\arrow[from=1-1, to=1-2]
	\arrow["i", from=1-2, to=1-3]
	\arrow["a"', from=1-2, to=2-2]
	\arrow["k", from=1-3, to=1-4]
	\arrow["m", from=1-3, to=2-3]
	\arrow[from=1-4, to=1-5]
	\arrow[equals, from=1-4, to=2-4]
	\arrow[from=2-1, to=2-2]
	\arrow["q"', from=2-2, to=2-3]
	\arrow["p"', from=2-3, to=2-4]
	\arrow[from=2-4, to=2-5]
\end{tikzcd}$$ Moreover, $\eta\in\mathcal{E}_F$ by \autoref{lema-morf-sucex}, since $q\in\mathcal{M}_F$ and $\mathcal{M}_F$ is an h.f.\,class. Now, by \autoref{prop-propiedades-pullback-pushout}, $[\eta]=[a\cdot\eta_1]$. Also, from \eqref{eq-diag-3x3} and \autoref{prop-propiedades-pullback-pushout}, $[a\cdot\eta_1]=[\eta_2\cdot g]$ and so $[\eta]=[\eta_2\cdot g]$. This means that there exists a commutative diagram with exact rows $$\begin{tikzcd}
	{\phantom{{}_2}\eta\colon\ \ \ 0} & B & M & H & 0\,\phantom{.} \\
	{\eta_2\colon\ \ \ 0} & B & E & I & 0\,.
	\arrow[from=1-1, to=1-2]
	\arrow["q", from=1-2, to=1-3]
	\arrow[from=2-1, to=2-2]
	\arrow["j"', from=2-2, to=2-3]
	\arrow["l"', from=2-3, to=2-4]
	\arrow[from=2-4, to=2-5]
	\arrow["p", from=1-3, to=1-4]
	\arrow[from=1-4, to=1-5]
	\arrow[Rightarrow, no head, from=1-2, to=2-2]
	\arrow["s"', from=1-3, to=2-3]
	\arrow["g", from=1-4, to=2-4]\arrow["\textup{\textsc{pb}}"{anchor=center}, draw=none, from=1-3, to=2-4]
\end{tikzcd}$$ Since $l,g\in\mathcal{M}_F$ and $g$ is monomorphism,  $s\in\mathcal{M}_F$ by \autoref{lema-popb-3x3}\,($b$). On the other hand, given that $ld=gk$ holds from \eqref{eq-diag-3x3}, universality of pullback provides $u\colon D\to M$ such that $k=pu$ and $d=su$. Moreover, $u$ is a monomorphism because so is $d$. If we are able to prove that $u\in\mathcal{M}_F$, given that $\mathcal{M}_F$ is an h.f.\,class and $s,u\in\mathcal{M}_F$ are monomorphisms, then $d=su\in\mathcal{M}_F$ by (E) and we are done. 

To see this, let us consider the following commutative diagram with exact rows, where $v\colon A\to B$ is the unique morphism such that $ui=qv$: $$\begin{tikzcd}
	{\eta_1\colon\ \ \ 0} & A & D & H & {0\,\phantom{.}} \\
	{\phantom{{}_1}\eta\colon\ \ \ 0} & B & M & H & {0\,.}
	\arrow[from=1-1, to=1-2]
	\arrow["i", from=1-2, to=1-3]
	\arrow[from=2-1, to=2-2]
	\arrow["q"', from=2-2, to=2-3]
	\arrow["p"', from=2-3, to=2-4]
	\arrow[from=2-4, to=2-5]
	\arrow["k", from=1-3, to=1-4]
	\arrow[from=1-4, to=1-5]
	\arrow["\exists !v"', dashed, from=1-2, to=2-2]
	\arrow["u", from=1-3, to=2-3]
	\arrow[Rightarrow, no head, from=1-4, to=2-4]
\end{tikzcd}$$ If we show that $v=a\in\mathcal{M}_F$, then an application of \autoref{lema-popb-3x3}\,($a$) will led us to conclude $u\in\mathcal{M}_F$ as desired. Indeed, from diagram \eqref{eq-diag-snake3x3} we have $k=pm$. On the other hand, we already have $k=pu$, thus $pu=pm$ implies $p(u-m)=0$ and so $u=qr+m$ for some $r\colon D\to B$, due to the fact that $q$ is a kernel of $p$. Now, from the relations $$di=sui=sqri+smi=jri+sqa=jri+ja=jri+di$$ we deduce $jri=0$ and so $ri=0$ because $j$ is a monomorphism. Hence, $ ui = qri + mi = qa$, and uniqueness of $v$ yields $v=a$.
\end{proof}

From the proof of \autoref{teo-principal} we see next that the definition of closed subfunctor is redundant. This result first appeared in \citep[Proposition 1.4] {draxler_exact_1999}.

\begin{coro}
For a proper subfunctor $F$ the following conditions are equivalent:
\begin{itemize}
\item[($a$)] $F$ is closed on the left.
\item[($b$)] $F$ is closed.
\item[($c$)] $F$ is closed on the right.
\end{itemize}
\end{coro}

\begin{proof}
By definition, ($b$) $\Rightarrow$ ($a$) and ($c$). Now if $F$ is closed on the left, then $\mathcal{M}_F$ satisfies (E) by \autoref{teo-caract-cerr-hf} and thus, as in the proof of \autoref{teo-principal}, $F$ has the $3\times 3$-lemma property (the only part where we use property (E$^{\ast}$) is for seeing that $s\in\mathcal{M}_F$, but this is a consequence of (E) as we properly mention in \autoref{obs-dual-lema-popb-3x3}). Therefore, $F$ is closed, which shows that ($a$) $\Rightarrow$ ($b$). Finally, ($c$) $\Rightarrow$ ($b$) follows by duality.
\end{proof}

We end this exposition by showing two applications of \autoref{teo-principal}. For this, we strongly use the fact that the $3\times 3$-lemma holds in abelian categories (see for instance \citep[Exercise 1.3.2]{weibel_introduction_1994}).

Let $\mathcal{X}\subseteq \A$ be a full subcategory. Then we obtain a subfunctor $F_{\mathcal{X}}$ by declaring the $F_{\mathcal{X}}$-exact sequences as those short exact sequences $$0\to A\to B\to C \to 0$$ for which the induced sequence $$0\to \hom{\A}{X}{A}\to \hom{\A}{X}{B}\to\hom{\A}{X}{C}\to 0$$ is exact in $\ab$ for every $X\in\mathcal{X}$. In a dual manner, we obtain a subfunctor $F^{\mathcal{X}}$. Both constructions gives rise to additive subfunctors \citep[Proposition 1.7]{auslander_relative1_1993}. Furthermore, they are closed \citep[Proposition 1.7]{draxler_exact_1999}.

\begin{prop}
For any full subcategory $\mathcal{X}$ of $\A$, the additive subfunctors $F_{\mathcal{X}}$ and $F^{\mathcal{X}}$ are closed.
\end{prop}

\begin{proof}
We only prove that $F_{\mathcal{X}}$ is closed since the proof of $F^{\mathcal{X}}$ uses dual arguments. Consider a commutative diagram 
\begin{equation}\label{eq-3x3-Fx}
\begin{tikzcd}
	& 0 & 0 & 0 \\
	0 & {A} & {B} & {C} & 0 \\
	0 & {D} & {E} & {G} & 0 \\
	0 & {H} & {I} & {J} & 0 \\
	& 0 & 0 & 0
	\arrow[from=2-1, to=2-2]
	\arrow[from=2-2, to=2-3]
	\arrow[from=2-3, to=2-4]
	\arrow[from=2-4, to=2-5]
	\arrow[from=3-1, to=3-2]
	\arrow[from=3-2, to=3-3]
	\arrow[from=3-3, to=3-4]
	\arrow[from=3-4, to=3-5]
	\arrow[from=4-1, to=4-2]
	\arrow[from=4-2, to=4-3]
	\arrow[from=4-3, to=4-4]
	\arrow[from=4-4, to=4-5]
	\arrow[from=1-2, to=2-2]
	\arrow[from=2-2, to=3-2]
	\arrow[from=3-2, to=4-2]
	\arrow[from=4-2, to=5-2]
	\arrow[from=1-3, to=2-3]
	\arrow[from=2-3, to=3-3]
	\arrow[from=3-3, to=4-3]
	\arrow[from=4-3, to=5-3]
	\arrow[from=1-4, to=2-4]
	\arrow[from=2-4, to=3-4]
	\arrow[from=3-4, to=4-4]
	\arrow[from=4-4, to=5-4]
\end{tikzcd}
\end{equation} such that the first and third row and all columns are $F_{\mathcal{X}}$-exact sequences. We shall show that the second row is also an $F_{\mathcal{X}}$-exact sequence. For every $X\in \mathcal{X}$, the functor $\hom{\A}{X}{-}$ applied to diagram \eqref{eq-3x3-Fx} yields the commutative diagram $$\begin{tikzcd}[column sep = 1.5em]
	& 0 & 0 & 0 \\
	0 & {\hom{\A}{X}{A}} & {\hom{\A}{X}{B}} & {\hom{\A}{X}{C}} & 0 \\
	0 & {\hom{\A}{X}{D}} & {\hom{\A}{X}{E}} & {\hom{\A}{X}{G}} & 0\\
	0 & {\hom{\A}{X}{H}} & {\hom{\A}{X}{I}} & {\hom{\A}{X}{J}} & 0 \\
	& 0 & 0 & 0
	\arrow[from=1-2, to=2-2]
	\arrow[from=1-3, to=2-3]
	\arrow[from=1-4, to=2-4]
	\arrow[from=2-1, to=2-2]
	\arrow[from=2-2, to=2-3]
	\arrow[from=2-2, to=3-2]
	\arrow[from=2-3, to=2-4]
	\arrow[from=2-3, to=3-3]
	\arrow[from=2-4, to=2-5]
	\arrow[from=2-4, to=3-4]
	\arrow[from=3-1, to=3-2]
	\arrow[from=3-2, to=3-3]
	\arrow[from=3-2, to=4-2]
	\arrow[from=3-3, to=3-4]
	\arrow[from=3-4, to=3-5]
	\arrow[from=3-3, to=4-3]
	\arrow[from=3-4, to=4-4]
	\arrow[from=4-1, to=4-2]
	\arrow[from=4-2, to=4-3]
	\arrow[from=4-2, to=5-2]
	\arrow[from=4-3, to=4-4]
	\arrow[from=4-3, to=5-3]
	\arrow[from=4-4, to=4-5]
	\arrow[from=4-4, to=5-4]
\end{tikzcd}$$ in which the first and third row and all columns are exact in $\ab$ by definition of $F_{\mathcal{X}}$. By the $3\times 3$-lemma, the second row of this diagram is also exact in $\ab$ and thus the second row of \eqref{eq-3x3-Fx} is an $F_{\mathcal{X}}$-exact sequence. Therefore, $F_{\mathcal{X}}$ has the $3\times 3$-lemma property and is closed by \autoref{teo-principal}.
\end{proof}

\begin{defi}
Let $F$ be an additive subfunctor. We say that:
\begin{itemize}
\item[$(a)$]  An object $P\in\A$ is \textit{a projective of $F$} if for every $F$-exact sequence $$0\to A\to B\to C\to 0\,,$$ the induced sequence of abelian groups $$0\to \hom{\A}{P}{A}\to\hom{\A}{P}{B}\to\hom{\A}{P}{C}\to 0$$ is exact. We denote by $\operatorname{Proj}(F)$ the full subcategory of $\A$ consisting of all projective objects of $F$.
\item[$(b)$] We say that $F$ \textit{has enough projectives} if for every $C\in\A$, there exists an $F$-exact sequence $$0\to A \to P\to C\to 0$$ where $P\in\operatorname{Proj}(F)$. 
\end{itemize}

Dually, we denote by $\operatorname{Inj}(F)$ the full subcategory of $\A$ consisting of the \textit{injective objects of $F$} and we define the notion of $F$ \textit{having enough injectives}.
\end{defi}

The next characterization is key for proving that subfunctors with enough projectives or injectives are closed.

\begin{prop}[{\citep[Proposition 1.5]{auslander_relative1_1993}}]\label{prop-suf-proj}
For an additive subfunctor $F$ the following conditions holds true.
\begin{itemize}
\item[$(a)$] If $F$ has enough projectives, then an exact sequence $$0\to A\to B\to C\to 0$$ is $F$-exact if, and only if, for every $P\in\operatorname{Proj}(F)$, the induced sequence of abelian groups $$0\to \hom{\A}{P}{A}\to\hom{\A}{P}{B}\to\hom{\A}{P}{C}\to 0$$ is exact.
\item[$(b)$] If $F$ has enough injectives, then an exact sequence $$0\to A\to B\to C\to 0$$ is $F$-exact if, and only if, for every $Q\in\operatorname{Inj}(F)$, the induced sequence of abelian groups $$0\to \hom{\A}{C}{Q}\to\hom{\A}{B}{Q}\to\hom{\A}{A}{Q}\to 0$$ is exact.
\end{itemize} \hfill $\qed$
\end{prop}

We now proceed to the second application of \autoref{teo-principal}, which corresponds to \citep[Theorem 7.1]{butler_classes_1961}.

\begin{prop}
If $F$ has enough projectives or injectives, then $F$ is closed.
\end{prop}

\begin{proof}
We only prove the case when $F$ has enough projectives. Consider a commutative diagram 
\begin{equation}\label{eq-3x3-suf-proj}
\begin{tikzcd}
	& 0 & 0 & 0 \\
	0 & {A} & {B} & {C} & 0 \\
	0 & {D} & {E} & {G} & 0 \\
	0 & {H} & {I} & {J} & 0 \\
	& 0 & 0 & 0
	\arrow[from=2-1, to=2-2]
	\arrow[from=2-2, to=2-3]
	\arrow[from=2-3, to=2-4]
	\arrow[from=2-4, to=2-5]
	\arrow[from=3-1, to=3-2]
	\arrow[from=3-2, to=3-3]
	\arrow[from=3-3, to=3-4]
	\arrow[from=3-4, to=3-5]
	\arrow[from=4-1, to=4-2]
	\arrow[from=4-2, to=4-3]
	\arrow[from=4-3, to=4-4]
	\arrow[from=4-4, to=4-5]
	\arrow[from=1-2, to=2-2]
	\arrow[from=2-2, to=3-2]
	\arrow[from=3-2, to=4-2]
	\arrow[from=4-2, to=5-2]
	\arrow[from=1-3, to=2-3]
	\arrow[from=2-3, to=3-3]
	\arrow[from=3-3, to=4-3]
	\arrow[from=4-3, to=5-3]
	\arrow[from=1-4, to=2-4]
	\arrow[from=2-4, to=3-4]
	\arrow[from=3-4, to=4-4]
	\arrow[from=4-4, to=5-4]
\end{tikzcd}
\end{equation} such that the first and third row and all columns are $F$-exact. We shall show that the second row is $F$-exact. According to \autoref{prop-suf-proj}, this is the case if for every $P\in \operatorname{Proj}(F)$, the sequence of abelian groups \begin{equation}\label{eq-hom(P,-)}
0 \to \hom{\A}{P}{D}\to\hom{\A}{P}{E}\to\hom{\A}{P}{G}\to 0
\end{equation} is exact.  Indeed, every $P\in\operatorname{Proj}(F)$ induces from \eqref{eq-3x3-suf-proj} a commutative diagram $$\begin{tikzcd}[column sep = 1.5em]
	& 0 & 0 & 0 \\
	0 & {\hom{\A}{P}{A}} & {\hom{\A}{P}{B}} & {\hom{\A}{P}{C}} & 0 \\
	0 & {\hom{\A}{P}{D}} & {\hom{\A}{P}{E}} & {\hom{\A}{P}{G}} & 0\\
	0 & {\hom{\A}{P}{H}} & {\hom{\A}{P}{I}} & {\hom{\A}{P}{J}} & 0 \\
	& 0 & 0 & 0
	\arrow[from=1-2, to=2-2]
	\arrow[from=1-3, to=2-3]
	\arrow[from=1-4, to=2-4]
	\arrow[from=2-1, to=2-2]
	\arrow[from=2-2, to=2-3]
	\arrow[from=2-2, to=3-2]
	\arrow[from=2-3, to=2-4]
	\arrow[from=2-3, to=3-3]
	\arrow[from=2-4, to=2-5]
	\arrow[from=2-4, to=3-4]
	\arrow[from=3-1, to=3-2]
	\arrow[from=3-2, to=3-3]
	\arrow[from=3-2, to=4-2]
	\arrow[from=3-3, to=3-4]
	\arrow[from=3-4, to=3-5]
	\arrow[from=3-3, to=4-3]
	\arrow[from=3-4, to=4-4]
	\arrow[from=4-1, to=4-2]
	\arrow[from=4-2, to=4-3]
	\arrow[from=4-2, to=5-2]
	\arrow[from=4-3, to=4-4]
	\arrow[from=4-3, to=5-3]
	\arrow[from=4-4, to=4-5]
	\arrow[from=4-4, to=5-4]
\end{tikzcd}$$ such that the first and third row and all columns are exact in $\ab$ by \autoref{prop-suf-proj} and then the $3\times 3$-lemma provides the exactness of sequence \eqref{eq-hom(P,-)}. Hence $F$ has the $3\times 3$-lemma property and is closed by \autoref{teo-principal}.
\end{proof}

\section*{Acknowledgments}
I would like to thank the anonymous referee for their meticulous reading and for providing detailed corrections and insightful suggestions, which significantly enhanced the clarity and quality of the final version.

%===========================%
%%%%		BIBLIOGRAPHY		%%%%

\bibliographystyle{siam} % Use the "unsrtnat" BibTeX style for formatting the Bibliography
%\nocite{*}
\bibliography{An_elementary_proof}

\end{document}